\documentclass[10pt]{article}
\usepackage{amsmath, amssymb, amsfonts, amsthm, epsfig, float, graphicx, sectsty}
\usepackage[all]{xy}
\title{Controlling LEF growth in some group extensions}
\author{Henry Bradford}

\newtheorem{thm}{Theorem}[section]
\newtheorem{lem}[thm]{Lemma}
\newtheorem{propn}[thm]{Proposition}
\newtheorem{coroll}[thm]{Corollary}
\newtheorem{defn}[thm]{Definition}
\newtheorem{ex}[thm]{Example}

\newtheorem{rmrk}[thm]{Remark}

\DeclareMathOperator{\Aut}{Aut}
\DeclareMathOperator{\Cay}{Cay}

\DeclareMathOperator{\diam}{diam}

\DeclareMathOperator{\FAlt}{FAlt}
\DeclareMathOperator{\FSym}{FSym}

\DeclareMathOperator{\im}{im}
\DeclareMathOperator{\PSL}{PSL}
\DeclareMathOperator{\Schr}{Schr}
\DeclareMathOperator{\SL}{SL}

\DeclareMathOperator{\supp}{supp}
\DeclareMathOperator{\Sym}{Sym}

\begin{document}

\maketitle

\begin{abstract}
%We study local embeddings into finite groups of certain 
%semidirect products, arising from two similar constructions. 
%For $\Gamma \curvearrowright \Omega$ a transitive action and $R$ 
%a commutative unital ring, we study $\FSym (\Omega) \rtimes \Gamma$ 
%and $E_{\Omega} (R) \rtimes \Gamma$, 
%where $\FSym (\Omega)$ is the group of all finitely supported permutations, 
%and  $E_{\Omega} (R)$ is the subgroup of $\Aut_R (R[\Omega])$ 
%generated by all transvections with respect to basis $\Omega$. 
%We estimate the size of finite groups affording local embeddings from balls 
%in the word-metric on such groups, as measured by the LEF growth function. 
%A key tool in the proof is to identify sequences of finitely presented 
%subgroups with short ``relative'' presentations. 
%As an application, we show that any sufficiently smooth increasing function 
%between $n!$ and $\exp(\exp(n))$ is close to the LEF growth function 
%of some finitely generated group. 
We study the LEF growth function of a finitely generated LEF group $\Gamma$, 
which measures the orders of finite groups admitting 
local embeddings of balls in a word metric on $\Gamma$. 
We prove that any sufficiently smooth increasing function 
between $n!$ and $\exp(\exp(n))$ is close to the LEF growth function 
of some finitely generated group. 
This is achieved by estimating the LEF growth of some semidirect products 
of the form $\FSym (\Omega) \rtimes \Gamma$, 
where $\Gamma \curvearrowright \Omega$ is an appropriate transitive action, 
and $\FSym (\Omega)$ is the group of finitely supported permutations of $\Omega$. 
A key tool in the proof is to identify sequences of finitely presented 
subgroups with short ``relative'' presentations. 
In a similar vein we also obtain estimates on the LEF growth 
of some groups of the form $E_{\Omega} (R) \rtimes \Gamma$, 
for $R$ an appropriate unital ring and 
$E_{\Omega} (R)$ the subgroup of $\Aut_R (R[\Omega])$ 
generated by all transvections with respect to basis $\Omega$. 
\end{abstract}

\section{Introduction}

A group $\Gamma$ is \emph{LEF} if every finite region of its 
multiplication table appears in the multiplication table of some finite group $Q$. 
The class of LEF groups properly contains 
the class of residually finite groups. 
If $\Gamma$ is finitely generated by a subset $S$, 
then the balls $B_S(n)$ in the associated word metric 
on $\Gamma$ form a prototypical sequence of finite subsets. 
The \emph{LEF growth function} $\mathcal{L}_{\Gamma} ^S$ of $\Gamma$ 
is defined such that $\mathcal{L}_{\Gamma} ^S (n) = \lvert Q \rvert$, 
where $Q$ is of minimal order among finite groups 
in which the (partial) multiplication on $B_S (n)$ may be modelled. 
LEF growth was introduced in \cite{ArzChe} 
and independently in \cite{BoRaStud} (under a different name). 
In this paper we continue our study of the LEF growth function 
initiated in \cite{Brad} and continued in \cite{BradDona}. 

\subsection*{Statement of results}

It is natural to ask which increasing functions 
arise as the LEF growth functions of groups. 
For instance in \cite{Brad} were exhibited 
groups with LEF growth on the order of $n^d$, $\exp(n^d)$ 
(for every positive integer $d$) 
and $\exp(\exp(n))$. 
In \cite{BradDona} it was further shown 
that there are uncountably many inequivalent types of LEF growth. 
We improve on these results, using completely different methods. 
Indeed we show that almost any sufficiently smoothly increasing function between 
exponential and doubly exponential is close to the LEF growth function of some group. 
Since the dependence of the function $\mathcal{L}_{\Gamma} ^S$ 
on the choice of finite generating set $S$ is slight, 
we suppress $S$ from our notation for the rest of the Introduction. 

\begin{thm} \label{SpectrumMainThm}
Let $f : \mathbb{N} \rightarrow \mathbb{N}$ 
be a permissible function. 
Then there exists a finitely generated LEF group $\Delta(f)$ 
and a constant $C>0$ such that 
$\mathcal{L}_{\Delta(f)} (n) \approx f(n)!$. 
\end{thm}

Here ``$\approx$'' denotes an appropriate notion of equivalence of nondecreasing 
functions (see Definition \ref{comparisondefn}). 
%Here ``$\preceq$'' denotes an appropriate notion of comparison of increasing 
%functions (see Definition \ref{comparisondefn}). 
The notion of a ``permissable'' function will be defined below 
(see Section \ref{SpectrumSect}), 
but suffice to say that it is flexible enough to yield the following examples, 
among many others. 
By ``$\preceq$'' we denote an appropriate notion of comparison of nondecreasing functions (see Definition \ref{comparisondefn}), 
such that $f_1 \approx f_2$ iff $f_1 \preceq f_2$ and $f_2 \preceq f_1$. 
%For nondecreasing functions $f$ and $g$ we write $f \approx g$ 
%if $f \preceq g$ and $g \preceq f$. 

\begin{coroll} \label{SpectrumCoroll}
For every $\alpha > 1$ and $\beta \in (0,1)$ 
there exist $C,c > 0$ and finitely generated LEF groups 
$\Gamma_1 (\alpha)$; $\Gamma_2 (\alpha)$ and $\Gamma_3 (\beta)$ such that: 
\begin{itemize}
\item[(i)] $\mathcal{L}_{\Gamma_1 (\alpha)}(n) \approx \exp (n^{\alpha})$; 
\item[(ii)] $\exp \big(c \exp(\log(n)^{\alpha})\big)
\preceq \mathcal{L}_{\Gamma_2 (\alpha)}(n) 
\preceq \exp \big(C \exp(\log(n)^{\alpha})\big)$; 
\item[(iii)] $\exp \big(c \exp(n^{\beta})\big)
\preceq \mathcal{L}_{\Gamma_3 (\beta)}(n) 
\preceq \exp \big(C \exp(n^{\beta})\big)$. 
\end{itemize} 
\end{coroll}

%but suffice to say that any function of the form $n^{\alpha}$ (for $\alpha \geq 1$); 
%$\exp (C \log (n)^{\beta})$ (for $C > 0$ and $\beta > 1$) 
%or $\exp(n^{\gamma})$ (for $0 < \gamma \leq 1$) is equivalent to 
%a permissable function. 

Let us begin to describe how the groups in Theorem \ref{SpectrumMainThm} 
are constructed. 
One attractive feature of the class of LEF groups, 
which is not shared by the class of residually finite groups, 
is that LEF is preserved under various natural classes of 
group extensions, while residual finiteness is not. 
To wit, for any groups $\Gamma$ and $\Delta$; 
a $\Gamma$-set $\Omega$, and any unital ring $R$, 
$\Gamma$ acts naturally by automorphisms on 
(i) the direct sum $\oplus_{\Omega} \Delta$; 
(ii) the group $\FSym(\Omega)$ of finitely supported permutations 
of the set $\Omega$, and 
(iii) the subgroup $E_{\Omega} (R)$ of $\Aut_R (R[\Omega])$ 
generated by all transvections with respect to $\Omega$ 
(where $R[\Omega]$ is the free $R$-module on basis $\Omega$). 
We therefore have corresponding semidirect products: 
(i) the \emph{wreath product} $\Delta \wr_{\Omega} \Gamma = (\oplus_{\Omega} \Delta) \rtimes \Gamma$; 
(ii) the \emph{symmetric enrichment} 
$\mathcal{S}(\Gamma,\Omega) = \FSym (\Omega) \rtimes \Gamma$, 
and (iii) the \emph{elementary enrichment} 
$\mathcal{E}(\Gamma,\Omega,R) 
= E_{\Omega} (R) \rtimes \Gamma$. 
If $\Gamma$ and $\Delta$ are finitely generated LEF groups 
and the action of $\Gamma$ on $\Omega$ is a transitive \emph{LEF action} 
(see Definition \ref{LEFactiondefn} below), 
then the groups $\Delta \wr_{\Omega} \Gamma$, 
$\mathcal{S}(\Gamma,\Omega)$, $\mathcal{E}(\Gamma,\Omega,\mathbb{Z})$ and 
$\mathcal{E}(\Gamma,\Omega,\mathbb{F}_p)$ are finitely generated LEF groups. 
In this paper we study the LEF growth of symmetric and elementary enrichments. 
To give a sample of our conclusions, in the special case for which 
$\Omega = \Gamma$ with the regular $\Gamma$-action, we have the following bounds 
(see Theorems \ref{RFextnThm} and \ref{elemthm} below for the full statements 
encompassing non-regular transitive LEF actions). 

\begin{thm} \label{RegSummaryThm}
Let $\Gamma$ be a finitely generated LEF group. 
\begin{itemize}
\item[(i)] $\gamma_{\Gamma}(n)! 
\preceq \mathcal{L}_{\mathcal{S}(\Gamma,\Gamma)} (n) 
\preceq \mathcal{L}_{\Gamma} (n)!$; 

\item[(ii)] $\exp\big(n \gamma_{\Gamma} (n^{1/2})^2 \big) \preceq 
\mathcal{L}_{\mathcal{E}(\Gamma,\Gamma,\mathbb{Z})} (n) 
\preceq \exp \big( n \mathcal{L}_{\Gamma}(n)^2 \big)$; 

\item[(iii)] $\exp\big(\gamma_{\Gamma} (n^{1/2})^2 /C \big) \preceq 
\mathcal{L}_{\mathcal{E}(\Gamma,\Gamma,\mathbb{F}_p)} (n) 
\preceq \exp \big(C \mathcal{L}_{\Gamma}(n)^2 \big)$ (for some $C>0$). 
\end{itemize}
\end{thm}

%Here $\gamma_{\Gamma}$ is the word growth function of $\Gamma$ and 
%``$\preceq$'' denotes an appropriate notion of comparison of increasing 
%functions (see Definition \ref{comparisondefn}). 
Here $\gamma_{\Gamma}$ is the word growth function of $\Gamma$. 
In particular, if $\mathcal{L}_{\Gamma}$ is comparable to $\gamma_{\Gamma}$, 
then $\mathcal{L}_{\mathcal{S}(\Gamma,\Gamma)}$ is known. 
Examples of groups $\Gamma$ for which this is the case 
include all finitely generated abelian groups and all non-virtually-nilpotent 
finitely generated linear groups. 

The groups $\Delta (f)$ arising in Theorem \ref{SpectrumMainThm} 
shall have the form $\mathcal{S}(\Gamma,\Omega(f))$, 
for $\Gamma$ a nonabelian free group and $\Omega(f)$ 
an infinite set which we equip with a transitive LEF $\Gamma$-action 
such that the growth function of the corresponding Schreier graph 
is precisely $f$. Our class of ``permissable'' 
functions will be defined so as to facilitate the construction of this action. 

\subsection*{Background and methods}

Local embeddings of elementary and symmetric enrichments into finite groups 
have been exploited in the context of several problems in 
geometric group theory. 
Vershik and Gordon \cite{VerGor} considered the group 
$\mathcal{S}(\Gamma,\Gamma)$ in the case $\Gamma = \mathbb{Z}$, 
and showed that it is LEF. 
Later Elek and Szabo \cite{EleSza} used symmetric enrichments 
to construct the first examples of sofic (indeed LEF) groups 
which are not residually amenable. 
Properties of the local embeddings of elementary enrichments 
were studied by Mimura and Sako in connection with 
constructions of exotic box-spaces in coarse geometry \cite{MimSak}, 
and were used in \cite{Mimu} to exhibit dense subgroups of a common profinite group 
with radically different coarse properties. 

Upper bounds on the LEF growth of wreath products; 
symmetric enrichments and elementary enrichments 
follow from the same ``natural'' elementary constructions 
which witness that these extensions preserve LEF. 
That is to say, if $Q_{\Gamma}$ and $Q_{\Delta}$ are finite groups, 
admitting local embeddings 
from appropriate subsets of $\Gamma$ and $\Delta$, 
and $q \in \mathbb{N}$ is sufficiently large, 
then balls in $\Delta \wr \Gamma$, $\mathcal{S}(\Gamma)$, 
$\mathcal{E}(\Gamma,\mathbb{Z})$ 
and $\mathcal{E}(\Gamma,\mathbb{F}_p)$ 
admit local embeddings into $Q_{\Delta} \wr Q_{\Gamma}$, 
$\mathcal{S}(Q_{\Gamma})$, $\mathcal{E}(Q_{\Gamma},\mathbb{Z}/q\mathbb{Z})$ 
and $\mathcal{E}(Q_{\Gamma},\mathbb{F}_p)$ respectively 
(and similarly for the extensions corresponding to non-regular LEF actions). 

More challenging are lower bounds on LEF growth. 
Our approach here exploits the equivalence 
of LEF and residual finiteness within the class of \emph{finitely presented} groups: 
if $\Gamma$ is finitely presented, then $\mathcal{L}_{\Gamma}$ 
is equivalent to the \emph{full residual finiteness growth function} 
$\mathcal{R}_{\Gamma}$. 
The latter has been extensively studied, and good estimates are known 
for many families of groups. 
For LEF groups which are \emph{not} finitely presented 
(as most of the examples we consider will not be), 
this equivalence can still be useful when applied on subgroups. 
In other words, 
by identifying a sequence $(H_n)$ of finitely presented 
subgroups of a group $\Gamma$, 
we extract useful lower bounds on $\mathcal{L}_{\Gamma}$ 
from lower bounds on $\mathcal{R}_{H_n}$ (at various radii; 
see Proposition \ref{LERFRFsubgrp} for the 
relevant technical statement). 
In symmetric and elementary enrichments, 
the groups $H_n$ will be respectively of the form 
$\Sym(k_n)$ and $\SL_{k_n} (\mathbb{Z})$ or $\SL_{k_n} (\mathbb{F}_p)$ 
(for appropriate increasing sequences of positive integers $(k_n)$). 

For the sake of comparison with Theorem \ref{RegSummaryThm}, 
an upper bound on the LEF growth of a regular wreath product 
of two LEF groups was proved in \cite{ArzChe}, 
and in \cite{Brad} a complementary lower bound was given in a special case. 

\begin{thm}[\cite{Brad} Theorem 1.9] \label{WPmainthm}
Let $\Gamma$ be a finitely generated LEF group 
and let $\Delta$ be a finite nontrivial group with trivial centre. Then: 
\begin{equation*}
\exp \big( \gamma_{\Gamma} (n) \big) \preceq \mathcal{L}_{\Delta \wr_{\Gamma} \Gamma}(n) 
\preceq \exp \big( \mathcal{L}_{\Gamma}(n) \big)\text{.}
\end{equation*}
\end{thm}

The lower bound in Theorem \ref{WPmainthm} 
uses different methods from the present paper, 
being based on the fact that one can recognize locally whether 
a family of finite centreless subgroups generates their direct product. 
We do not treat the case of the wreath product $\Delta \wr_{\Omega} \Gamma$ 
in detail here, although the methods of \cite{ArzChe} and \cite{Brad} would allow one 
to generalize the bounds of Theorem \ref{WPmainthm} to these groups 
for non-regular LEF actions, 
at least when $\Delta$ is finite centreless. 

The paper is structured as follows: 
in Section \ref{PrelimSect} detail basic properties of the LEF growth function
discuss the relationship between LEF growth and full residual 
finiteness growth, and connect this relationship 
to finite presentability. 
We also record some 
finite presentations of $\Sym(n)$, $\SL_n(\mathbb{Z})$ and $\SL_n(p)$, 
to be used subsequently. 
In Sections \ref{FinExtSect} and \ref{ElemExtSect} 
we prove our bounds on the LEF growth of 
symmetric enrichments and elementary enrichments of LEF groups, respectively. 
In Section \ref{SpectrumSect}, we prove Theorem \ref{SpectrumMainThm}.

\subsection*{Basic notation and terminology}

An action of a group $\Gamma$ on a set $\Omega$ will be denoted 
$\Gamma \curvearrowright \Omega$. 
Unless specified all our actions will be on the right. 
For $S \subseteq \Gamma$ the associated \emph{Schreier graph} 
$\Schr (\Gamma,\Omega,S)$ is the directed graph with vertices $\Omega$ 
and edges $\Omega \times S$, with 
$\iota (\omega , s) = \omega$ and $\tau (\omega , s)=\omega s$. 
We regard the edge $(\omega , s)$ as being labelled by the element $s\in S$. 
For $n \in \mathbb{N}$ and $\omega \in \Omega$, 
$B_{S,\omega}(n) \subseteq \Omega$ denotes the closed ball of radius 
$n$ around $\omega$ in the path-metric on $\Schr (\Gamma,\Omega,S)$. 
We shall also denote by $B_{S,\omega}(n)$ 
the induced subgraph of $\Schr (\Gamma,\Omega,S)$ on this set. 
In the special case of $\Omega = \Gamma$ with the regular action, 
$\Schr (\Gamma,\Omega,S)=\Cay(\Gamma,S)$ is the \emph{Cayley graph} 
and we write $B_S(n)$ for $B_{S,1}(n)$. 

\section{Preliminaries} \label{PrelimSect}

This Section summarises much of the basic material which was treated 
in greater detail in \cite{Brad}[Section 2]. 
There is also signicicant overlap with the relevant parts of \cite{ArzChe}. 

\subsection{Definition and first properties}

\begin{defn}
For $\Gamma,\Delta$ groups and $F \subseteq \Gamma$, 
a \emph{partial homomorphism} of $F$ into $\Delta$ is a 
function $\pi : F \rightarrow \Delta$ 
such that, for all $g,h \in F$, if $gh \in F$, 
then $\pi(gh)=\pi(g)\pi(h)$. 
$\pi$ is called a \emph{local embedding} if it is injective. 
$\Gamma$ is \emph{locally embeddable into finite groups (LEF)} 
if, for all finite $F \subseteq \Gamma$, 
there exists a finite group $Q$ and a local embedding of 
$F$ into $Q$. 
\end{defn}

\begin{rmrk} \label{restrmrk}
If $F^{\prime} \subseteq F \subseteq \Gamma$ 
and $\pi : F \rightarrow \Delta$ is a local embedding, 
then so is $\pi \mid_{F^{\prime}} : F^{\prime} \rightarrow \Delta$. 
\end{rmrk}

Henceforth suppose that $\Gamma$ is generated by the finite set $S$. 
\begin{defn}

The \emph{LEF growth} of $\Gamma$ (with respect to $S$) is: 
\begin{center}
$\mathcal{L}_{\Gamma} ^S (n) = \min \lbrace \lvert Q \rvert 
: \exists \pi :B_S(n)\rightarrow Q\text{ a local embedding} \rbrace$
\end{center}
(with $\mathcal{L}_{\Gamma} ^S (n) = \infty$ 
if the set on the right-hand side is empty). 
\end{defn}

\begin{defn} \label{comparisondefn}
\normalfont
Given functions $f_1 , f_2 : (0,\infty) \rightarrow (0,\infty)$ 
we write $f_1 \preceq f_2$ if there exists a constant $C>0$ 
such that $f_1(x) \leq f_2 (Cx)$ for all $x \in (0,\infty)$. 
We say \emph{$f_1$ is equivalent to $f_2$} and write 
$f_1 \approx f_2$ if $f_1 \preceq f_2$ 
and $f_2 \preceq f_1$. 
We may also compare functions which are defined only on 
$\mathbb{N}$ under $\preceq$ and $\approx$: 
for this purpose $f : \mathbb{N}\rightarrow (0,\infty)$ 
may be extended to $(0,\infty)$ by declaring 
$f$ to be constant on each interval $(n,n+1]$. 
\end{defn}

\begin{lem} \label{subgrplem}
Let $\Delta \leq \Gamma$ be finitely generated by $S^{\prime}$. 
Let $M \in \mathbb{N}$ be such that $S^{\prime} \subseteq B_S (M)$. 
Then for all $n$, $\mathcal{L}_{\Delta} ^{S^{\prime}} (n) \leq \mathcal{L}_{\Gamma} ^S (Mn)$. 
In particular $\mathcal{L}_{\Delta} ^{S^{\prime}} \preceq \mathcal{L}_{\Gamma} ^S$. 
\end{lem}

Consequently, the LEF growth only depends on a choice of 
generating set up to equivalence. 

\begin{coroll}
Let $S,S^{\prime} \subseteq \Gamma$ be finite generating sets. 
Let $L,M \in \mathbb{N}$ be such that $S^{\prime} \subseteq B_S (L)$ 
and $S \subseteq B_{S^{\prime}} (M)$. 
Then: 
\begin{center}
$\mathcal{L}_{\Gamma} ^S (n) \leq \mathcal{L}_{\Gamma} ^{S^{\prime}} (Mn) \leq \mathcal{L}_{\Gamma} ^S (LMn)$. 
\end{center}
In particular 
$\mathcal{L}_{\Gamma} ^S \approx \mathcal{L}_{\Gamma} ^{S^{\prime}}$. 
\end{coroll}

\begin{defn} \label{LEFactiondefn}
For $\Gamma \curvearrowright \Omega$ and $\Delta \curvearrowright X$ 
group actions and $F \subseteq \Gamma$ and $\Sigma \subseteq \Omega$ 
finite subsets, a \emph{local embedding} of $(F,\Sigma)$ 
into $(\Delta,X)$ is a pair $(\pi,\theta)$ of a local embedding 
$\pi : F \rightarrow \Delta$ and an injective map 
$\theta : \Sigma \rightarrow X$ satisfying: 
\begin{equation} \label{LEFactioneqn}
\text{for all } g \in F \text{ and } \omega \in \Sigma, 
\text{ if } \omega g\in\Sigma \text{ then } \theta(\omega g)=\theta(\omega)\pi(g). 
\end{equation}
The group action $\Gamma \curvearrowright \Omega$ is 
\emph{locally embeddable into finite actions (LEF)} if, 
for all finite subsets 
$F \subseteq \Gamma$ and $\Sigma \subseteq \Omega$, 
there exists a finite action $Q \curvearrowright X$ 
and a local embedding of $(F,\Sigma)$ into $(Q,X)$. 
\end{defn}

\begin{ex}
If $F \subseteq \Gamma$ and $\pi : F \rightarrow \Delta$ 
is a local embedding, then $(\pi,\pi)$ 
is a local embedding of $(F,F)$ into $(\Delta,\Delta)$ 
(with respect to the regular actions 
$\Gamma \curvearrowright \Gamma$ and 
$\Delta \curvearrowright \Delta$). 
As such, if $\Gamma$ is a LEF group, 
then the regular action $\Gamma \curvearrowright \Gamma$ 
is a LEF action. 

More generally, if $H \leq \Gamma$ is closed in the profinite 
topology on $\Gamma$ and has trivial normal core, 
then for any finite $F \subseteq \Gamma$ and 
$\Sigma \subseteq H\backslash \Gamma$, 
there exists $H \leq \overline{H} \leq \Gamma$, 
with $\lvert \Gamma:H\rvert < \infty$, 
such that the natural maps 
$\Gamma\rightarrow \Sym(\overline{H}\backslash \Gamma)$ 
and $H\backslash \Gamma\rightarrow \overline{H}\backslash \Gamma$ 
induce a local embedding of $(F,\Sigma)$ into 
$(\Sym(\overline{H}\backslash \Gamma),\overline{H}\backslash \Gamma)$. 
Thus $\Gamma \curvearrowright H\backslash \Gamma$ is 
a LEF action. 
\end{ex}

\begin{ex} \label{LERFex}
Suppose $\Gamma$ is a split extension with kernel 
$K$ and quotient $Q$. 
If $\Gamma$ is LEF then the action of 
$Q$ on $N$ by conjugation is LEF. 
\end{ex}

A further interesting construction of LEF actions will be given in 
Remark \ref{SymEnrRmrks} (ii) below. 

\subsection{LEF, residual finiteness and finite presentability}

Let $\Gamma$ be a finitely generated group and 
$S \subseteq \Gamma$ be a finite generating set. 
Recall that $\Gamma$ is \emph{residually finite (RF)} if, 
for every finite subset $F \subset \Gamma$, 
there exists a finite group $Q$ and a homomorphism 
$\pi : \Gamma \rightarrow Q$ such that the restriction 
$\pi |_F$ of $\pi$ to $F$ is injective. 
Following \cite{BoRaMcR}, one may quantify this definition as follows. 

\begin{defn}
The \emph{(full) RF growth} of $\Gamma$ is: 
\begin{center}
$\mathcal{R}_{\Gamma} ^S (n) 
= \min \lbrace \lvert F \rvert 
: \exists \pi : \Gamma \rightarrow F \text{ s.t. } 
\pi |_{B_S(n)} \text{ is injective} \rbrace$. 
\end{center}
\end{defn}

The next observation is almost immediate from these definitions. 

\begin{lem} 
If $\Gamma$ is residually finite, then it is LEF, and
$\mathcal{L}_{\Gamma} ^S (n) 
\leq \mathcal{R}_{\Gamma} ^S (n)$ for all $n$. 
\end{lem}

Among finitely \emph{presented} groups we have a converse inequality. 

\begin{propn} \label{LEFRFpropn}
If $\Gamma$ is finitely presented and LEF then it is residually finite. 
Moreover, if $\Gamma \cong \langle S \mid R \rangle$ 
is a finite presentation with $R \subseteq F(S)$ 
consisting of reduced words of length at most $n_0$, 
and $\mathcal{C}$ is a class of finite groups, 
then: $\mathcal{R}_{\Gamma,\mathcal{C}} ^S (n) 
= \mathcal{L}_{\Gamma,\mathcal{C}} ^S (n)$
for all $n \geq n_0$. 
\end{propn}

\begin{coroll} 
Suppose $\mathcal{R}_{\Gamma} ^S (n) \neq  \mathcal{L}_{\Gamma} ^S (n)$ for infinitely many $n$. Then $\Gamma$ is not 
finitely presentable. 
\end{coroll}

The following ``relative'' variant of Proposition \ref{LEFRFpropn} 
will be one of our main tools 
for proving non-trivial lower bounds on LEF growth in examples. 

\begin{propn} \label{LERFRFsubgrp}
Let $\Gamma$ be finitely generated by $S$. 
For $n \in \mathbb{N}$ let $H_n \leq \Gamma$, 
and let $S^{\prime} _n \subseteq S^{\prime \prime} _n 
\subseteq H_n$ be finite subsets. 
Suppose that $H_n$ is finitely presented 
by $\langle S^{\prime} _n \mid R_n \rangle$, 
and let $f,g : \mathbb{N} \rightarrow \mathbb{N}$ 
be such that for all $n \in \mathbb{N}$: 
\begin{align*}
S^{\prime \prime} _n & \subseteq B_S (f(n));  \\
R_n & \subseteq B_{S^{\prime} _n} (g(n))\text{ (in $F(S^{\prime} _n)$).} 
\end{align*}
Then $\mathcal{L}_{\Gamma} ^S (f(n)g(n)) 
\geq \mathcal{R}_{H_n} ^{S^{\prime\prime} _n} (g(n))$.
\end{propn}

\begin{proof}
Viewing $S^{\prime} _n$ as a subset of $F(S)$, 
we have $R_n \subseteq B_S (f(n)g(n))$. 
Let $Q$ be a finite group and 
$\pi : B_S (f(n)g(n)) \hookrightarrow Q$ 
be a local embedding. 
Then since $e_{\Gamma} \in B_S (f(n)g(n))$, 
$\pi(R_n) = \lbrace e_Q \rbrace$, 
so $\pi \mid_{S^{\prime} _n}$ extends to a homomorphism 
$\tilde{\pi} : H_n \rightarrow Q$, 
which agrees with $\pi$ on 
$B_{S^{\prime\prime} _n} (g(n)) \subseteq B_S (f(n)g(n))$. 
In particular, the restriction of $\tilde{\pi}$ 
to $B_{S^{\prime\prime} _n} (g(n))$ is injective, 
so $\lvert Q \rvert \geq \mathcal{R}_{H_n} ^{S^{\prime\prime} _n} (g(n))$. 
\end{proof}

\subsection{LEF growth and growth of Schreier graphs}

Let $\Gamma$ be finitely generated by $S$, 
and suppose $\Gamma \curvearrowright \Omega$. 
For $\omega_0 \in \Omega$, the \emph{growth function} 
$\gamma_{\Gamma,\Omega} ^{S,\omega_0} : \mathbb{N}\rightarrow \mathbb{N}$ 
of the associated Schreier graph $\Schr(\Gamma,\Omega,S)$ 
(with respect to $S$ and $\omega_0$)
is given by 
$\gamma_{\Gamma,\Omega} ^{S,\omega_0} (n)= \lvert B_{S,\omega_0}(n)\rvert$. 
If $\Gamma \curvearrowright \Omega$ is transitive, 
it is not difficult to see that the $\approx$-class of 
$\gamma_{\Gamma,\Omega} ^{S,\omega_0}$ is independent of $S$ and $\omega_0$; 
we may write $\gamma_{\Gamma,\Omega}$ for this class. 
In the special case of the regular action $\Gamma \curvearrowright \Gamma$, 
we recover the \emph{word growth} function $\gamma_{\Gamma} ^S$ 
of $\Gamma$ (with respect to $S$); 
we write $\gamma_{\Gamma}$ for its $\approx$-class. 

\begin{rmrk} \label{growthlowerbdrmrk}
If there is a local embedding $B_S (n) \rightarrow Q$, 
then $\lvert B_S (n) \rvert \leq \lvert Q \rvert$, 
so $\gamma_{\Gamma} ^S (n) \leq \mathcal{L}_{\Gamma} ^S (n)$. 
\end{rmrk}

\begin{defn}
The group $\Gamma$ is 
\emph{efficiently locally embeddable in finite groups} (ELEF) 
if $\gamma_{\Gamma} \approx \mathcal{L}_{\Gamma}$. 
\end{defn}

\begin{ex} \label{ELEFListEx}
\normalfont
\begin{itemize}
\item[(i)] $\mathbb{Z}^k$ is ELEF; $\mathcal{L}_{\mathbb{Z}^k}(n) \approx n^k$ 
\cite{Brad}[Example 2.24]; 

\item[(ii)] A finitely generated virtually nilpotent group is ELEF iff 
it is virtually abelian \cite{Brad}[Proposition 2.26]; 

\item[(iii)] Let $\Gamma$ be a finitely generated group admitting 
a faithful linear representation over a field, 
and suppose $\Gamma$ is not virtually nilpotent. 
Then it well-known that $\gamma_{\Gamma}$ is exponential. 
Meanwhile by \cite{BouCorn}[Theorem 1.1] 
$\mathcal{R}_{\Gamma}$ (called in that paper 
\emph{residual girth}) is at most exponential. 
It follows that $\Gamma$ is ELEF, 
and $\mathcal{L}_{\Gamma} (n) \approx \exp (n)$ 
\cite{Brad}[Example 2.25]; 

\item[(iv)] Since every nonabelian finite-rank free group $F$ is linear over 
$\mathbb{Z}$ and is not virtually nilpotent, 
$\Gamma=F$ satisfies the conclusion of (iii). 

\end{itemize}
\end{ex}

\subsection{Group presentations}

In order to extract good lower bounds on $\mathcal{L}_{\Gamma} ^S$ 
from Proposition \ref{LERFRFsubgrp} in Sections \ref{FinExtSect} 
and \ref{ElemExtSect} below, 
it will be important that the subgroups $H_n$ 
admit presentations with relator-sets consisting of short words. 

\subsubsection{The symmetric group}

The presentations of the finite symmetric groups, 
with a generating set consisting of transpositions, 
were studied in \cite{Solo}, 
where the following Theorem was proved. 

\begin{thm} \label{Symtreepresthm}
Let $T=(V,E)$ be a finite tree. 
For $e = (v,w) \in E$, let $s_{e} = (v \; w) \in \Sym(V)$. 
Then $S_T = \lbrace s_{e} : e \in E \rbrace$ 
generates $\Sym(V)$, with defining presentation 
$\big\langle S_T \mid R_T \big\rangle$, 
where $R_T = R_T ^1 \cup R_T ^2 \cup R_T ^3 \cup R_T ^4$ and: 
\begin{align*}
R_T ^1 & = \lbrace s_e ^2 : e \in E \rbrace \\
R_T ^2 & = \lbrace (s_{(v,w)} s_{(x,y)})^2 : (v,w),(x,y) \in E; 
v,w,x,y \in V \text{ distinct} \rbrace \\
R_T ^3 & = \lbrace (s_{(v,w)} s_{(v,x)})^3 
: (v,w),(v,x) \in E; w \neq x \rbrace \\
R_T ^4 & = \lbrace (s_{(v,w)} s_{(v,x)} s_{(v,w)} s_{(v,y)})^2 
: (v,w),(v,x),(v,y) \in E; v,w,x,y \in V \text{ distinct} \rbrace\text{.}
\end{align*}
\end{thm}

Note that the special case of Theorem \ref{Symtreepresthm} for which 
$T$ is a line-segment is precisely the Coxeter presentation for $\Sym(V)$. 

\begin{rmrk} \label{symgenrmrk}
It is clear that $\lbrace s_e : e \in E \rbrace$ 
generates $\Sym (V)$. 
For let $v , w \in V$. 
Let $v = v_0 , v_1 , \ldots , v_l = w$ be a path in $T$ 
from $v$ to $w$. Then: 
\begin{center}
$(v \; w) = s_{(v_1,v_2)} ^{s_{(v_0,v_1)}\cdots s_{(v_{l-1},v_l)}}$ 
\end{center}
so all transpositions are generated by 
$\lbrace s_e : e \in E \rbrace$, as required. 
Moreover it is not difficult to see that the relations 
$R_T$ between the $S_T$ hold in $\Sym(V)$. 
To prove Theorem \ref{Symtreepresthm}, 
it therefore suffices to check that $R_T$ is a sufficient set of 
relations between the generators $S_T$. 
\end{rmrk}

\subsubsection{The special linear group} \label{SLsubsubsect}

Let $R$ be a commutative ring with unity. 
For $V$ a countable set, let $R^V$ be the free $R$-module 
(with $V$ a free basis). 
Given $v,w\in V$ distinct elements and $r \in R$, 
define the \emph{transvection} $E_{v,w} (r) \in \Aut_R (R^V)$ by: 
\begin{center}
$E_{v,w} (r): v\mapsto v+rw$ and 
$E_{v,w} (r): x\mapsto x$ (for $x \neq v$). 
\end{center}
For $v,w,x,y \in V$ and $r,s \in R$, we compute: 
\begin{equation*}
[E_{v,x}(r),E_{x,w}(s)] = E_{v,w}(rs) \text{ if }v \neq w
\end{equation*}
while for $v \neq y$ and $w \neq x$, 
$E_{v,w}(r)$ and $E_{x,y}(s)$ commute. 

Write $E_{v,w} = E_{v,w} (1)$ for short. 
Let $E_V (R) \leq \Aut_R(R^V)$ be the group generated by 
all transvections $E_{v,w} (r)$, 
as $(v,w)$ ranges over all pairs of distinct elements in $V$ 
and $r$ ranges over $R$. 

\begin{thm} \label{SLpresnthm}
Suppose $V$ is finite with $\lvert V \rvert \geq 3$. 
Then $E_V (\mathbb{Z}) \cong \SL_{V} (\mathbb{Z})$ 
is generated by $S_V = \lbrace E_{v,w} :v,w \in V, v \neq w \rbrace$ 
and admits a presentation $\big\langle S_V \mid R_V \big\rangle$, 
where: 
\begin{align*}
R_V = & \big\lbrace [E_{v,x},E_{x,w}]E_{v,w} ^{-1} : v \neq w \big\rbrace \\
& \cup \big\lbrace [E_{v,w},E_{x,y}] : v\neq y,w\neq x \big\rbrace \\
& \cup \lbrace (E_{v_0 , w_0} E_{w_0 , v_0} ^{-1} E_{v_0 , w_0})^4 \rbrace \text{.}
\end{align*}
for any (fixed) distinct $v_0 , w_0 \in V$. 

Similarly, for $p$ prime, 
$E_V (\mathbb{F}_p) \cong \SL_{V} (\mathbb{F}_p)$ 
is generated by $S_V$, and is presented by 
$\big\langle S_V \mid R_V ^p \big\rangle$, 
where $R_V ^p = R_V \cup \lbrace E_{v_0,w_0} ^p \rbrace$. 
\end{thm}

\begin{proof}
This is covered in, for instance, the Introduction to \cite{CoRoWi}. 
In this treatment, $V$ is identified with $\lbrace 1,\ldots,n\rbrace$ 
(for $n=\lvert V\rvert$) and $(v_0,w_0)$ is taken to be $(1,2)$. 
There is no loss of generality here, since all $E_{v,w}$ 
are conjugate in $\SL_{V} (\mathbb{Z})$. 
\end{proof}

Now let $G = (V,E)$ be a finite connected graph. 
Recall that the \emph{diameter} $\diam(G)$ 
of $G$ is the maximal distance between pairs of vertices 
in $G$ in the path metric. 
Let $S_{G} = \lbrace E_{v,w},E_{w,v} : (v,w) \in E \rbrace \subseteq \SL_V(\mathbb{Z})$. 

\begin{coroll} \label{SLgraphpresncoroll}
Let $G$ and $S_G$ be as above. 
Then $\SL_V(\mathbb{Z})$ admits a presentation 
$\big\langle S_G \mid R_G \big\rangle$ where 
$R_G \subseteq F(S_G)$ consists of words of length 
$O\big( \diam(G)^2 \big)$. 
\end{coroll}

\begin{proof}
Since $\SL_V(\mathbb{Z}) \cong \big\langle S_V \mid R_V \big\rangle$ 
as in Theorem \ref{SLpresnthm}, 
with $R_V \subseteq F(S_V)$ consisting of words of bounded length, 
it suffices to show that $S_V \subseteq B_{S_G} (M)$ 
for some $M = O \big( \diam(G)^2 \big)$. 
Let $v,w \in V$ be distinct, and let $p = (v=v_0,v_1,\ldots,v_n=w)$ 
be a path of minimal length in $G$ from $v$ to $w$. 
We claim that, if $n \leq 2^m$, then $E_{v,w} \in B_{S_G}(4^m)$. 
Certainly the claim holds for $m=0$. 
Assuming the claim for smaller $m$, 
let $x = v_{\lceil n/2 \rceil}$. 
Then $x$ is at distance at most $2^{m-1}$ from both $v$ and $w$, 
so $E_{v,x},E_{x,w} \in B_{S_G}(4^{m-1})$ by induction, 
and we have $E_{v,w} = [E_{v,x},E_{x,w}]$. 
The result then follows, since every pair $v,w$ admits a path 
$p$ of length $n \leq \diam(G)$. 
\end{proof}

\section{Symmetric enrichments} \label{FinExtSect}

Let $\FSym (\Omega)$ be the group of finitely supported 
permutations of the set $\Omega$. 
Note that: 
\begin{itemize}
\item[(i)] $\FSym (\Omega) \vartriangleleft \Sym (\Omega)$; 

\item[(ii)] For any finite subset $\Sigma \subseteq \Omega$, 
we have $\Sym(\Sigma) \leq \FSym (\Omega)$ 
(naturally identifying $\Sym(\Sigma)$ with the set of permutations 
of $\Omega$ supported on $\Sigma$). 
\end{itemize}

Let $\Gamma$ be a group acting transitively on the set $\Omega$. 
The action induces a homomorphism 
$\tilde{\phi} : \Gamma \rightarrow \Sym (\Omega)$ 
which in turn induces 
$\phi : \Gamma \rightarrow \Aut \big( \FSym (\Omega) \big)$. 

\begin{defn}
The \emph{symmetric enrichment} of the action 
$\Gamma \curvearrowright \Omega$ 
is the group $\mathcal{S}(\Gamma,\Omega) = \FSym (\Omega) \rtimes_{\phi} \Gamma$. 
If $\Omega = \Gamma$, equipped with the regular $\Gamma$-action, 
then we may write $\mathcal{S}(\Gamma)$ 
for $\mathcal{S}(\Gamma,\Gamma)$. 
\end{defn}

\begin{rmrk} \label{AltEnrRmrk}
There is also $\FAlt(\Omega) \leq \FSym(\Omega)$, 
the index-$2$ subgroup consisting of finitely-supported 
even permutations of $\Omega$. 
$\FAlt(\Omega)$ is normalized by the action of $\Gamma$ via $\phi$, 
so we may also define the \emph{alternating enrichment} 
$\mathcal{A}(\Gamma,\Omega) 
= \FAlt (\Omega)\rtimes_{\phi} \Gamma \leq \mathcal{S}(\Gamma,\Omega)$ 
and $\mathcal{A}(\Gamma) \leq \mathcal{S}(\Gamma)$. 
\end{rmrk}

\begin{thm} \label{RFextnThm}
Let $\Gamma$ be a LEF group, finitely generated by 
the symmetric set $S$, 
and let $\Gamma \curvearrowright \Omega$ be a transitive LEF action. 
Then $\mathcal{S}(\Gamma,\Omega)$ is finitely generated and LEF. 
Moreover, if $(\pi_n,\theta_n)$ is a sequence of local embeddings  
of $\big(B_S(n), B_{S,\omega_0}(n)\big)$ into $(Q_n,X_n)$ 
then: 
\begin{equation} \label{RFextneqn1}
\gamma_{\Gamma,\Omega}(n)! 
\preceq \mathcal{L}_{\mathcal{S}(\Gamma,\Omega)} (n) 
\preceq \lvert X_n \rvert ! \lvert Q_n \rvert
\end{equation}
In particular, for any finitely generated LEF group $\Gamma$, then: 
\begin{equation} \label{RFextneqn2}
\gamma_{\Gamma}(n)! 
\preceq \mathcal{L}_{\mathcal{S}(\Gamma)} (n) 
\preceq \mathcal{L}_{\Gamma} (n)!
\end{equation}
\end{thm}

\begin{coroll} \label{FinEnrELEFCoroll}
If $\Omega$ is infinite then 
$\mathcal{L}_{\mathcal{S}(\Gamma,\Omega)}(n) \succeq n!$. 
In particular $\mathcal{S}(\Gamma,\Omega)$ is not ELEF. 
\end{coroll}

\begin{proof}
This follows from (\ref{RFextneqn1}) and the fact that every infinite 
Schreier graph of a 
finitely generated group has growth at least linear and at most exponential. 
\end{proof}

\begin{coroll} \label{FinEnrFPCoroll}
If $\Omega$ is infinite 
and $\Gamma \curvearrowright \Omega$ is a transitive LEF action, 
then $\mathcal{S}(\Gamma,\Omega)$ is not finitely presentable. 
\end{coroll}

\begin{proof}
By Theorem \ref{RFextnThm} and Proposition \ref{LEFRFpropn}, 
it suffices to prove that $\mathcal{S}(\Gamma,\Omega)$ is not residually finite. 
This is so because $\mathcal{S}(\Gamma,\Omega)$ contains $\FAlt(\Omega)$, 
an infinite simple group. 
\end{proof}

We imagine that the conclusion of Corollary \ref{FinEnrFPCoroll} 
holds under much weaker conditions on the action, 
and may be known to some experts 
(in a related vein see the criteria for finite presentability 
of a wreath product in \cite{Corn}). 
We are however not aware of an existing reference. 

\begin{ex}
If $k \geq 1$ and $\Gamma = \mathbb{Z}^k$ then 
by Example \ref{ELEFListEx} (i) 
$\mathcal{L}_{\Gamma}(n) \approx \gamma_{\Gamma} (n) \approx n^k$, 
so by (\ref{RFextneqn2}), 
$\mathcal{L}_{\mathcal{S}(\Gamma)} (n) \approx (n^k)!$. 
\end{ex}

\begin{ex}
For $k \geq 2$, if either $\Gamma = F_k$ 
or $\Gamma = \SL_k(\mathbb{Z})$ then 
by Example \ref{ELEFListEx} (iii) and (iv) 
$\mathcal{L}_{\Gamma}(n) \approx \gamma_{\Gamma} (n) \approx \exp(n)$, 
so by (\ref{RFextneqn2}), 
$\mathcal{L}_{\mathcal{S}(\Gamma)} (n) \approx (\exp (n))!$. 
\end{ex}

We now prove Theorem \ref{RFextnThm}. 
The first step is to show finite generation of $\mathcal{S}(\Gamma,\Omega)$. 

\begin{lem} \label{permtreediamlem}
Let $S$ be a symmetric generating set for $\Gamma$ and let 
$\omega_0\in\Omega$. 
Let $G_n = (V_n,E_n)$ be the finite graph with 
vertex-set $V_n =B_{S,\omega_0}(n)$ and edge-set 
$E_n=\lbrace(\omega_0 g,\omega_0 sg):s\in S,g\in B_S (n-1)\rbrace$. 
Then $G_n$ is a connected graph, and contains a maximal tree $T_n$ 
of diameter at most $2n$. 
\end{lem}

\begin{proof}
We proceed by induction on $n$. 
The base case is $G_1$, which is the star-graph with leaves 
$\omega_0 S \setminus \lbrace \omega_0 \rbrace$. 
For $n \geq 2$, and given $T_{n-1}$, 
we choose for all $v \in V_n \setminus V_{n-1}$ 
an $n$-tuple $(s_1 , \ldots , s_n)$ such that 
$v = \omega_0 s_1 \cdots s_n$. 
Writing $s_v = s_1$, $g_v = s_2 \cdots s_n$, 
we have $(\omega_0 g_v ,v)=(\omega_0 g_v ,\omega_0 s_v g_v) \in E_n$ and $\omega_0 g_v \in V_{n-1}$. 
Define the graph $T_n$ on $V_n$ by 
$E(T_n) = E(T_{n-1}) \cup \lbrace (\omega_0 g_v ,v) 
: v \in V_n \setminus V_{n-1} \rbrace$. 
Then $T_n$ is a tree on $V_n$ (since $T_{n-1}$ is a tree 
on $V_{n-1}$ and the elements of $V_n \setminus V_{n-1}$ 
are leaves of $T_n$). Moreover every element of  
$V_n \setminus V_{n-1}$ is at distance $1$ in $T_n$ 
from an element of $V_{n-1}$, 
so $\diam(T_n) \leq \diam(T_{n-1}) + 2 \leq 2n$. 
\end{proof}

\begin{lem} \label{FinExtGenLem}
If $\Gamma$ is finitely generated by the symmetric set $S$ 
then for any $\omega_0 \in \Omega$, 
$\mathcal{S}(\Gamma,\Omega)$ is finitely generated by 
$S \cup T_{S,\omega_0}$, where: 
\begin{center}
$T_{S,\omega_0} = \lbrace (\omega_0 \;\omega_0 s) : s \in S \rbrace \subseteq \FSym (\Gamma)$. 
\end{center}
\end{lem}

\begin{proof}
Let $\sigma \in \FSym(\Omega)$ and $g \in \Gamma$. 
Certainly $g \in \langle S \rangle$. 
By transitivity of $\Gamma \curvearrowright \Omega$, 
there exists $n \in \mathbb{N}$ such that 
$\supp(\sigma) \subseteq \Sym \big( B_{S,\omega_0}(n) \big)$. 
Let $T_n$ be as in Lemma \ref{permtreediamlem}. 
By Remark \ref{symgenrmrk}, 
$\Sym \big( B_{S,\omega_0}(n) \big)$ is generated 
by $S_{T_n} = \lbrace s_e : e \in E(T_n) \rbrace$ 
as in the statement of Theorem \ref{Symtreepresthm}. 
For $e \in E(T_n)$ there exist $h \in \Gamma$ and $s \in S$ 
such that $s_e = (\omega_0 h \; \omega_0 sh) = (\omega_0 \; \omega_0 s)^h 
\in \langle S \cup T_{S,\omega_0} \rangle$. 
\end{proof}

The upper bound in (\ref{RFextneqn1}) will follow from 
the next construction. 

\begin{lem} \label{SDprodgenlem}
Let $\Gamma$, $N$ be groups, 
let $\phi : \Gamma \rightarrow \Aut (N)$ be a homomorphism, 
and let $S \subseteq \Gamma$, $T \subseteq N$ be symmetric subsets. 
Then within $N \rtimes_{\phi} \Gamma$, 
\begin{equation}
B_{S \cup T} (n) \subseteq B_{C(S,T,n)} (n) \cdot B_S (n)
\end{equation}
where $C(S,T,n) = \lbrace t^w : w \in B_S (n-1),t\in T \rbrace 
\subseteq N$. 
\end{lem}

\begin{proof}
For $g \in B_{S \cup T} (n)$, 
there exist $s_1 , \ldots , s_n \in S \cup \lbrace 1 \rbrace$ 
and $t_1 , \ldots , t_n \in T \cup \lbrace 1 \rbrace$ 
such that $g = t_1 s_1 \cdots t_n s_n = h\cdot (s_1 \cdots s_n)$, 
where: 
\begin{equation*}
h = t_1 \big( t_2 ^{s_1 ^{-1}} \big)
\big( t_3 ^{(s_1 s_2) ^{-1}} \big) 
\cdots \big( t_n ^{(s_1 \cdots s_{n-1}) ^{-1}} \big)
\end{equation*}
and for $1 \leq k \leq n$, 
$t_k ^{(s_1 \cdots s_{k-1})^{-1}} \in C(S,T,n)$. 
\end{proof}

\begin{propn} \label{RFextnprop}
Let $I(n)=\Sym\big( B_{S,\omega_0}(n)\big) \cdot B_S (n)
\subseteq \mathcal{S}(\Gamma,\Omega)$ 
and let $(\pi_n,\theta_n)$ be a local embedding 
of $\big(B_S(n), B_{S,\omega_0}(n)\big)$ into $(Q_n,X_n)$. 
Then there is a local embedding of $I(n)$ into $\mathcal{S}(Q_n,X_n)$. 
\end{propn}

\begin{proof}
Write $B(n) = B_{S,\omega_0}(n)$. 
Let $\tilde{\theta}_n : \Sym\big( B(n)\big) 
\rightarrow \Sym(X_n)$ be the (injective) homomorphism 
induced by the injection 
$\theta_n : B(n) \rightarrow X_n$; that is: 
\begin{center}
$[\theta_n (\omega)]\tilde{\theta}_n (\sigma) = \theta_n (\omega \sigma)$ 
and $[x]\tilde{\theta}_n (\sigma) = x$ 
for $x \in X_n \setminus \im (\theta_n)$. 
\end{center}
Let $\Phi_n : I(n) \rightarrow \mathcal{S}(Q_n,X_n)$ be given by 
$\Phi_n (\sigma \cdot g)=\tilde{\theta}_n (\sigma) \cdot \pi_n (g)$. 
$\Phi_n$ is clearly injective; 
we claim that it is a local embedding. 

For $\sigma_1 ,\sigma_2 \in\Sym\big( B(n)\big)$, 
$g_1 , g_2 \in B_S (n)$, we have $(\sigma_1 g_1) (\sigma_2 g_2) 
= \sigma_1 \sigma_2 ^{g_1 ^{-1}} g_1 g_2 \in I(n)$ 
iff (i) $g_1 g_2 \in B_S (n)$ and (ii) 
$\supp (\sigma_1 \sigma_2 ^{g_1 ^{-1}}) 
\subseteq B(n)$. 
Since $\supp (\sigma_1) \subseteq B(n)$, 
(ii) is equivalent to 
(ii') $\supp (\sigma _2) \subseteq B(n) g_1 \cap B(n)$. 
Under these conditions, 
\begin{align*}
\Phi_n \big( (\sigma_1 g_1) (\sigma_2 g_2) \big) &= 
\Phi_n \big( \sigma_1 \sigma_2 ^{g_1 ^{-1}} g_1 g_2 \big) \\
&= \tilde{\theta}_n \big( \sigma_1 \sigma_2 ^{g_1 ^{-1}} \big) \pi_n (g_1 g_2) \\
& = \tilde{\theta}_n (\sigma_1) \tilde{\theta}_n (\sigma_2 ^{g_1 ^{-1}}) \pi_n (g_1)\pi_n (g_2) \text{ (by (i))}
\end{align*}
while: 
\begin{align*}
\Phi_n (\sigma_1 g_1)\Phi_n (\sigma_2 g_2)
& = \tilde{\theta}_n (\sigma_1) \pi_n (g_1) 
\tilde{\theta}_n (\sigma_2) \pi_n (g_2) \\
& = \tilde{\theta}_n (\sigma_1)  
\tilde{\theta}_n (\sigma_2)^{\pi_n (g_1)^{-1}} \pi_n (g_1)\pi_n (g_2)
\end{align*}
so, writing $\tau_1 = \tilde{\theta}_n (\sigma_2)^{\pi_n (g_1)^{-1}}, \tau_2 = \tilde{\theta}_n (\sigma_2 ^{g_1 ^{-1}}) \in \Sym(X_n)$, 
it suffices to check that $\tau_1 = \tau_2$. 
Since, by (ii'), $\supp (\sigma_2) 
\subseteq B(n) g_1 \cap B(n)$, 
\begin{center}
$\supp (\tilde{\theta}_n (\sigma_2)) 
 \subseteq \theta_n (B(n)) \pi_n (g_1) \cap \theta_n (B(n))$; 
$\supp (g_1 \sigma_2 g_1 ^{-1}) \subseteq B(n) g_1 ^{-1} \cap B(n)$
\end{center}
so that: 
\begin{equation} \label{tau1eqn}
\supp ( \tau_1 )
\subseteq \theta_n (B(n))\pi_n (g_1)^{-1} \cap \theta_n (B(n))
\end{equation}
and: 
\begin{align}
\supp (\tau_2) 
& \subseteq \theta_n (B(n)g_1 ^{-1} \cap B(n)) \label{tau2eqn1} \\
& \subseteq \theta_n (B(n))\pi_n (g_1)^{-1} \label{tau2eqn3}
\cap \theta_n (B(n))
\end{align} 
We verify that for all $x \in X_n$, $(x)\tau_1 = (x)\tau_2$. 
First suppose $x\in\theta_n (B(n)g_1 ^{-1}\cap B(n))$, 
so that there are $\omega , \tilde{\omega} \in B(n)$ 
with $\omega = \tilde{\omega} g_1 ^{-1}$ are such that 
$x = \theta_n (\omega) = \theta_n (\tilde{\omega} g_1 ^{-1})$ 
so $x = \theta_n (\tilde{\omega}) \pi_n (g_1 ^{-1})$. 
Thus: 
\begin{center}
$\tau_1 : x \mapsto 
\theta_n (\tilde{\omega}) \tilde{\theta}_n (\sigma_2) \pi_n (g_1)^{-1}$
\end{center} 
while: 
\begin{center}
$\tau_2 :x \mapsto 
\theta_n (\tilde{\omega} \sigma_2 g_1 ^{-1}) 
= \theta_n (\tilde{\omega} \sigma_2) \pi_n (g_1 ^{-1})
=\theta_n(\tilde{\omega})\tilde{\theta}_n(\sigma_2)\pi_n (g_1 ^{-1})$
\end{center}
(as $\supp (\sigma_2) \subseteq B(n)g_1 \cap B(n)$, 
so either $\tilde{\omega} \notin \supp (\sigma_2)$ 
or $\tilde{\omega}\sigma_2,\tilde{\omega}\sigma_2g_1 ^{-1}\in B(n)$). 

Next suppose  $x \in X_n \setminus \theta_n (B(n))$. 
Then by (\ref{tau1eqn}) and (\ref{tau2eqn3}), 
$\tau_1$ and $\tau_2$ fix $x$. 
Finally suppose $x\in\theta_n(B(n))\setminus\theta_n(B(n)g_1 ^{-1})$. 
Then by (\ref{tau2eqn1}) $\tau_2$ fixes $x$. 
If in addition $x \notin \theta_n (B(n)) \pi_n (g_1)^{-1}$, 
then by (\ref{tau1eqn}) $\tau_1$ fixes $x$ as well. 
Therefore suppose that there exist $\omega,\tilde{\omega} \in B(n)$ 
such that 
$x =\theta_n (\omega) =\theta_n (\tilde{\omega})\pi_n(g_1)^{-1}$. 
Then $\tilde{\omega} g_1 ^{-1} \notin B(n)$ 
(else $\theta_n (\tilde{\omega} g_1 ^{-1})
=\theta_n (\tilde{\omega})\pi_n(g_1)^{-1}
=x=\theta_n (\omega)$, so $\omega = \tilde{\omega}g_1 ^{-1}$ 
and $x \in \theta_n (B(n)g_1 ^{-1})$, a contradiction). 
Thus $\tilde{\omega} \notin B(n)g_1 \supseteq \supp(\sigma_2)$, so: 
\begin{align*} 
\tau_1 : x = \theta_n (\tilde{\omega})\pi_n (g_1)^{-1} 
& \mapsto \theta_n (\tilde{\omega})\tilde{\theta}_n (\sigma_2)\pi_n (g_1)^{-1} \\
& = \theta_n (\tilde{\omega}\sigma_2) \pi_n (g_1)^{-1} \\
& = \theta_n (\tilde{\omega}) \pi_n (g_1)^{-1} \\
& = x
\end{align*}
\end{proof}

\begin{proof}[Proof of Theorem \ref{RFextnThm}]
Let $I(n)$ be as in Proposition \ref{RFextnprop}. 
To prove the upper bound in (\ref{RFextneqn1}), 
it suffices to show that 
$B_{S \cup T_{S,\omega_0}} (n) \subseteq I(n)$. 
Taking $T = T_{S,\omega_0}$ in Lemma \ref{SDprodgenlem}, 
it suffices to show that for 
$t \in T_{S,\omega_0}$ and $w \in B_S (n-1)$, 
$t^w \in \Sym \big( B_{S,\omega_0}(n) \big)$.  
This is so, since letting $s \in S$ 
be such that $t = (\omega_0 \; \omega_0 s)$, 
so that $\supp \big( t ^{w} \big) 
= \lbrace \omega_0 w,\omega_0 s w \rbrace 
\subseteq B_{S,\omega_0}(n)$. 

For the lower bound, we apply Proposition \ref{LERFRFsubgrp}, 
replacing $\Gamma$ by $\mathcal{S}(\Gamma,\Omega)$ 
and $S$ by $S \cup T_{S,\omega_0}$ as in Lemma \ref{FinExtGenLem}, 
and taking $H_n = \Sym\big( B_{S,\omega_0}(n)\big)$. 
Let the graph $G_n$ and the tree $T_n$ be as in Lemma \ref{permtreediamlem}. 
Let $H_n = \big\langle S_{T_n} \mid R_{T_n} \big\rangle$ 
be as in Theorem \ref{Symtreepresthm}, 
and set $S^{\prime\prime} _n = S^{\prime} _n = S_{T_n}$. 
By the description of $R_{T_n}$ from Theorem \ref{Symtreepresthm}, 
we can take $g(n) = 8$ for all $n$. 
Further, by the definition of the graph $G_n$ 
from Lemma \ref{permtreediamlem} 
and the generating set $S_{T_n}$ from Theorem \ref{Symtreepresthm}, 
for any $s^{\prime\prime} \in S^{\prime\prime} _n = S_{T_n}$, 
there exist $h \in B_S (n-1)$ and $s \in S$ such that: 
\begin{center}
$s^{\prime\prime} = (\omega_0 h \; \omega_0 sh) = (\omega_0 \; \omega_0 s)^h \in B_{S \cup T_{S,\omega_0}} (2n-1)$. 
\end{center}
We may therefore take $f(n) = 2n-1$ in Proposition \ref{LERFRFsubgrp} 
and obtain that for all $n$, 
\begin{center}
$\mathcal{L}_{\mathcal{S}(\Gamma,\Omega)} ^{S\cup T_{S,\omega_0}}(16n-8) 
\geq \mathcal{R}_{\Sym( B_{S,\omega_0}(n))} ^{S_{T_n}} (8)$. 
\end{center}
The right-hand side of the last inequality is the order of a 
quotient of $\Sym\big( B_{S,\omega_0}(n)\big)$, 
into which $B_{S_{T_n}} (8)$ injects. 
For $n \geq 5$, $\lvert B_{S,\omega_0}(n) \rvert \geq 5$ 
so $\Sym\big( B_{S,\omega_0}(n)\big)$ has no proper 
quotients of order greater than $2$, 
while $\lvert B_{S_{T_n}} (8)\rvert \geq 8$, so: 
\begin{center}
$\mathcal{R}_{\Sym( B_{S,\omega_0}(n))} ^{S_{T_n}} (8) 
= \big\lvert \Sym\big( B_{S,\omega_0}(n)\big) \big\rvert
=  \gamma_{\Gamma,\Omega} ^S (n)!$. 
\end{center}
\end{proof}

\begin{rmrk} \label{SymEnrRmrks}
\begin{itemize}
\item[(i)] The conclusions of Theorem \ref{RFextnThm} are also valid 
for the alternating enrichment $\mathcal{A}(\Gamma,\Omega)$. 
The upper bounds hold simply because 
$\mathcal{A}(\Gamma,\Omega) \leq \mathcal{S}(\Gamma,\Omega)$. 
The lower bounds follow by applying 
Proposition \ref{LERFRFsubgrp} to the presentations 
of the alternating groups, induced from those 
appearing in Theorem \ref{Symtreepresthm} for the symmetric groups 
by the Reidemeister-Schreier procedure. 

\item[(ii)] For $\Gamma \curvearrowright \Omega$ we have an 
induced action $\mathcal{S}(\Gamma,\Omega) \curvearrowright \Omega$. 
A straightforward extension of the proof of 
Proposition \ref{RFextnprop} shows that 
if the former action is LEF then so is the latter 
(with effective control of the order of the finite action 
required to admit a local embedding of balls). 
\end{itemize}
\end{rmrk}

\section{Elementary enrichments} \label{ElemExtSect}

Let $E_{v,w}(r)$, $E_{v,w} = E_{v,w}(1)$ and $E_V (R)$ 
be as in \ref{SLsubsubsect}. 
For $W \subseteq V$ there is a natural embedding 
$E_W (R) \leq E_V (R)$, induced by acting as the identity on 
all basis vectors in $V \setminus W$. 
If $X$ is the ascending union of finite subsets 
$V_n \subseteq V$, then: 
\begin{equation*}
E_V (R) = \cup_n E_{V_n} (R)
\end{equation*}  
Now let $\Gamma$ be a group, acting transitively on the set $V$. 
$\Gamma$ acts by automorphisms on $E_V (R)$ via: 
\begin{equation*}
(E_{v,w}(r))^g = E_{vg,wg}(r)
\end{equation*}
We may therefore define the \emph{$R$-elementary enrichment} 
of the action $\Gamma \curvearrowright V$ over $R$ to be the group: 
\begin{center}
$\mathcal{E}(\Gamma,V,R) = E_V (R) \rtimes \Gamma$. 
\end{center}
In case $V = \Gamma$ and the action is regular, 
we write $\mathcal{E} (\Gamma,R)$ for short. 

Passing from $\Gamma$ to $\mathcal{E}(\Gamma,V,R)$ 
does not in general preserve residual finiteness. 

\begin{propn} \label{elemnotRFprop}
Suppose $\Gamma$ is finitely generated and $V$ is infinite. 
Then $\mathcal{E}(\Gamma,V,R)$ 
is not residually finite. 
\end{propn}

We imagine the conclusion of Proposition \ref{elemnotRFprop} 
to be known to experts but, being unable to locate a proof in the 
literature, we include one here. 

\begin{lem} \label{elemoneinlem}
Suppose $\lvert V \rvert \geq 6$. 
Let $N \vartriangleleft \mathcal{E}(\Gamma,V,R)$. 
Suppose there exist $v \neq w$ such that $E_{v,w} \in N$. 
Then $E_V (R) \leq N$. 
\end{lem}

\begin{proof}
Let $x,y \in V$ be distinct and let $r \in R$. 
Claim that $E_{x,y} (r) \in N$. 
Let $t,u \in V \setminus \lbrace v,w,x,y \rbrace$ be distinct. Then: 
\begin{center}
$E_{t,u}(r) = \big[ E_{t,v}(r) ,[E_{v,w},E_{w,u}] \big] \in N$
\end{center}
so $E_{x,y}(r) = \big[ E_{x,t} ,[E_{t,u}(r),E_{u,y}] \big] \in N$. 
\end{proof}

\begin{proof}[Proof of Proposition \ref{elemnotRFprop}]
Let $N \vartriangleleft_f \mathcal{E}(\Gamma,V,R)$. 
We claim that $E_V (R) \leq N$, from which the conclusion follows. 
By Lemma \ref{elemoneinlem} it suffices to show that there 
exist distinct $v,w \in V$ such that $E_{v,w} \in N$. 

If there exists $v \in V$ and $g \in N \cap \Gamma$ 
such that $v,v^g$ $v^{g^2}$ are distinct, then:
\begin{align*}
1 = [E_{v,v^g}, E_{v,v^g}] & \equiv [E_{v,v^g}, g^{-1}E_{v,v^g} g] \mod N \\
& = [E_{v,v^g},E_{v^g,v^{g^2}}] \\
& = E_{v,v^{g^2}}
\end{align*}
so $E_{v,v^{g^2}} \in N$, as required. 
Thus we may assume that for all $v \in V$ and $g \in N \cap \Gamma$, 
$v^{g^2} = v$. It follows that, 
where $\rho : \Gamma \rightarrow \Sym(V)$ is the homomorphism 
induced by $\Gamma \curvearrowright V$, $\rho (N \cap \Gamma)$ 
is of exponent at most $2$, hence is abelian. 
Moreover $\rho (N \cap \Gamma)$ is finitely generated, 
(since $N \cap \Gamma \leq_f \Gamma$ is), 
so $\rho (N \cap \Gamma)$ is a finite group. 
This is a contradiction, 
since  by transitivity of $\Gamma \curvearrowright V$, 
$\rho (\Gamma)$ is infinite, 
hence so is $\rho (N \cap \Gamma) \leq_f \rho (\Gamma)$. 
\end{proof}

By contrast $\mathcal{E}(\Gamma,\Omega,R)$ 
\emph{can} admit many local embeddings. 
Here we emphasize the cases $R=\mathbb{Z}$ and $\mathbb{F}_p$. 
First we must deal with the issue of finite generation. 

\begin{propn} \label{elemFGprop}
Let $\Gamma$ be finitely generated by 
the symmetric set $S$ and let 
$\Gamma \curvearrowright \Omega$ be a transitive action. 
Let: 
\begin{center}
$T(\omega_0) = \lbrace E_{\omega_0,\omega_0 s} : s \in S \rbrace 
\cup \lbrace E_{\omega_0 s,\omega_0} : s \in S \rbrace 
\subseteq E_{\Omega} (R)$.  
\end{center}
Then for all $\omega , \psi \in \Omega$ distinct, 
$E_{\omega , \psi} \in \langle S \cup T(\omega_0) \rangle$. Moreover, 
there exists an absolute constant $C>0$ such that 
for all $n \in \mathbb{N}$ and $g \in \Gamma$, if $\lvert g \rvert_S \leq n$ then: 
\begin{equation} \label{elemFGeqn}
\big\lvert E_{\omega_0 , \omega_0 g} \big\rvert_{S \cup T(\omega_0)} , 
\big\lvert E_{\omega_0 g, \omega_0} \big\rvert_{S \cup T(\omega_0)} \leq Cn^2 .
\end{equation}
\end{propn}

\begin{proof}
We can reduce to the case $\omega = \omega_0$ or $\psi = \omega_0$, 
since otherwise: 
\begin{center}
$E_{\omega , \psi} = \big[ E_{\omega , \omega_0} , E_{\omega_0 , \psi} \big]$. 
\end{center}
Further, since $\Gamma \curvearrowright \Omega$ is transitive, 
we may write $\omega = \omega_0 g$ and $\psi = \omega_0 h$ 
for some $g,h \in \Gamma$. Thus, if $\omega \neq \omega_0$ we have: 
\begin{center}
$E_{\omega ,\omega_0}=E_{\omega_0 g ,\omega_0}=(E_{\omega_0 ,\omega_0 g^{-1}} )^g$
\end{center} 
so it suffices to prove (\ref{elemFGeqn}). 
To this end, we prove, by induction on $m \in \mathbb{N}$, that: 
\begin{equation} \label{elemFGeqn2}
L_m = 
\max_{\lvert g \rvert_S \leq 2^m , \omega_0 g \neq \omega_0}\lvert E_{\omega_0 , \omega_0 g} \rvert_{S \cup T(\omega_0)} \leq (C/2) 4^m . 
\end{equation}
The corresponding bound on 
$\lvert E_{\omega_0 g, \omega_0} \rvert_{S \cup T(\omega_0)}$ 
follows by symmetry 
and the inequality (\ref{elemFGeqn}) for general $n$ follows 
since $2^m \leq n \leq 2^{m+1}$ for some $m$. 
To this end take $g \in B_S (2^m)$ with $\omega_0 g \neq \omega_0$. 
If there exists $s \in S$ such that $\omega_0 g = \omega_0 s$ 
then $E_{\omega_0 , \omega_0 g} \in T(\omega_0)$. 
This covers the case $m=0$. 
Assume therefore that this is not the case. 
Then there exist $h , k \in \Gamma$ with 
$\lvert h \rvert_S , \lvert k \rvert_S \leq 2^{m-1}$ 
such that $g = hk$ and $\omega_0$, $\omega_0 k$ and $\omega_0 g$ are all distinct 
(so that $\omega_0 \neq \omega_0 h$ as well). 
We have: 
\begin{equation*}
E_{\omega_0 , \omega_0 g} 
= \big[ E_{\omega_0 , \omega_0 k} , E_{\omega_0 k , \omega_0 g} \big] 
= \big[ E_{\omega_0 , \omega_0 k} , (E_{\omega_0, \omega_0 h})^k \big] 
\end{equation*}
so that: 
\begin{align*}
\lvert E_{\omega_0 , \omega_0 g} \rvert_{S \cup T(\omega_0)} 
& \leq 2 \lvert E_{\omega_0 , \omega_0 h} \rvert_{S \cup T(\omega_0)}
+ 2 \lvert E_{\omega_0 , \omega_0 k} \rvert_{S \cup T(\omega_0)}
+ 4 \lvert k \rvert_S \\
& \leq 4 L_{m-1} + 2^{m+1}
\end{align*}
Taking the maximum over all such $g$, $L_m \leq 4 L_{m-1} + 2^{m+1}$, 
and solving the corresponding recurrence yields (\ref{elemFGeqn2}). 

%It suffices to show that for all $g,h \in \Gamma$ such that 
%$\omega_0 g \neq \omega_0 h$, 
%$E_{\omega_0 g,\omega_0 h} \in \langle S \cup T(\omega_0) \rangle$. 
%Let $n \in \mathbb{N}$ be such that $g,h \in B_S (n)$. 
%Let $G_n = (V_n,E_n)$ be as in Lemma \ref{permtreediamlem}. 
%Let $\omega_0 g = v_0 , v_1 , \ldots , v_k = \omega_0 h$ 
%be a reduced path in $G_n$. 
%We proceed by induction on $k$. 
%If $k = 1$ then there exists $s \in S$ such that either 
%$E_{\omega_0 g,\omega_0 h} 
%= E_{\omega_0 g,\omega_0 sg}=(E_{\omega_0,\omega_0 s} )^g$ 
%or $E_{\omega_0 g,\omega_0 h} 
%= E_{\omega_0 sh,\omega_0 h}=(E_{\omega_0 s,\omega_0} )^h$, 
%and in both cases 
%$E_{\omega_0 g,\omega_0 h} \in \langle S \cup T(\omega_0) \rangle$. 
%If $k \geq 2$, then $\omega_0 g \neq v_{k-1} \neq \omega_0 h$; 
%$E_{\omega_0 g , v_{k-1}} , E_{v_{k-1} , \omega_0 h} 
%\in \langle S \cup T(\omega_0) \rangle$ (by induction), 
%and: 
%\begin{center}
%$E_{\omega_0 g,\omega_0 h}
%=[E_{\omega_0 g,v_{k-1}},E_{v_{k-1},\omega_0 h}] 
%\in \langle S \cup T(\omega_0) \rangle$. 
%\end{center}
\end{proof}

\begin{coroll} \label{elemFGcoroll}
Let $R=\mathbb{Z}$ or $\mathbb{F}_p$. 
Let $\Gamma$, $S$, $\Omega$ and $T(\omega_0)$ 
be as in Proposition \ref{elemFGprop}. 
Then $\mathcal{E}(\Gamma,\Omega,R)$ is finitely generated 
by $S \cup T(\omega_0)$. 
\end{coroll}

\begin{proof}
$\mathcal{E}(\Gamma,\Omega,R)$ is generated by $\Gamma$ 
and $E_{\Omega} (R)$. 
$\Gamma$ is generated by $S$, and by Proposition \ref{elemFGprop} 
$E_{\Omega} (R) \subseteq \langle S \cup T(\omega_0) \rangle$. 
\end{proof}

Now we can bound the LEF growth of $\mathcal{E}(\Gamma,\Omega,\mathbb{Z})$ 
and $\mathcal{E}(\Gamma,\Omega,\mathbb{F}_p)$. 
Let $T(\omega_0)$ be as in Proposition \ref{elemFGprop}. 

\begin{thm} \label{elemthm}
Let $\Gamma$ be a LEF group, finitely generated by $S$, 
and let $\Gamma \curvearrowright \Omega$ be a transitive LEF action. 
Then $\mathcal{E}(\Gamma,\Omega,\mathbb{Z})$ 
and $\mathcal{E}(\Gamma,\Omega,\mathbb{F}_p)$ are LEF. 
Moreover, if $(\pi_n,\theta_n)$ is a sequence of local embeddings  
of $\big(B_S(n), B_{S,\omega_0}(n)\big)$ into $(Q_n,X_n)$ 
then there exists $C > 0$ such that: 
\begin{equation} \label{elemthmgeneqn}
\exp\big(n \gamma_{\Gamma,\Omega} ^{S,\omega_0}(n^{1/2})^2 / C \big)
\leq \mathcal{L}_{\mathcal{E}(\Gamma,\Omega,\mathbb{Z})} ^{S\cup T(\omega_0)} (n) 
\leq \exp \big( Cn \lvert X_n \rvert^2 \big) \lvert Q_n \rvert\text{.}
\end{equation}
\begin{equation} \label{p-elemthmgeneqn}
\exp\big(\gamma_{\Gamma,\Omega} ^{S,\omega_0}(n^{1/2})^2 /C \big)
\leq \mathcal{L}_{\mathcal{E}(\Gamma,\Omega,\mathbb{F}_p)} ^{S\cup T(\omega_0)} (n) 
\leq \exp \big( C\lvert X_n \rvert^2 \big) \lvert Q_n \rvert\text{.}
\end{equation}
In particular, for any finitely generated LEF group $\Gamma$, 
\begin{equation} \label{elemthmregeqn}
\exp\big(n \gamma_{\Gamma} (n^{1/2})^2 \big) \preceq 
\mathcal{L}_{\mathcal{E}(\Gamma,\mathbb{Z})} (n) 
\preceq \exp \big( n \mathcal{L}_{\Gamma}(n)^2 \big)\text{.}
\end{equation}
\begin{equation} \label{p-elemthmregeqn}
\exp\big(\gamma_{\Gamma} (n^{1/2})^2/C \big) \preceq 
\mathcal{L}_{\mathcal{E}(\Gamma,\mathbb{F}_p)} (n) 
\preceq \exp \big(C \mathcal{L}_{\Gamma}(n)^2 \big)\text{.}
\end{equation}
\end{thm}

The next observation follows from Theorem \ref{elemthm} 
much as Corollary \ref{FinEnrELEFCoroll} 
follows from Theorem \ref{RFextnThm}. 

\begin{coroll}
If $\Omega$ is infinite then $\mathcal{E}(\Gamma,\Omega,\mathbb{Z})$ is not $ELEF$. 
\end{coroll}

\begin{coroll}
If $\Gamma$ is a LEF group and $\Omega$ is infinite, 
then $\mathcal{E}(\Gamma,\Omega,\mathbb{Z})$ 
and $\mathcal{E}(\Gamma,\Omega,\mathbb{F}_p)$ 
are not finitely presentable. 
\end{coroll}

\begin{proof}
Let $R = \mathbb{Z}$ or $\mathbb{F}_p$. 
If $\Gamma$ is finitely generated, 
then by Proposition \ref{elemnotRFprop} and Theorem \ref{elemthm}, 
$\mathcal{E}(\Gamma,\Omega,R)$ is LEF but not 
residually finite, so by Proposition \ref{LEFRFpropn} 
is not finitely presentable. 
If $\Gamma$ is \emph{not} finitely generated, 
then neither is $\mathcal{E}(\Gamma,\Omega,R)$ 
(since the former is a quotient of the latter). 
\end{proof}

For the lower bound in (\ref{elemthmgeneqn}) 
we shall need the congruence subgroup property for 
$\SL_V (\mathbb{Z})$ with $3 \leq \lvert V \rvert < \infty$. 
For $q \in \mathbb{N}$ let 
$\pi_q : \SL_V(\mathbb{Z}) \rightarrow 
\SL_V(\mathbb{Z}/q\mathbb{Z})$ be the \emph{congruence homomorphism} 
(given by $\pi_q (A_{v,w}) = A_{v,w} \mod q$). 
If $r \in \mathbb{N}$ with $q \mid r$ 
then we also denote by $\pi_q$ the map 
$\SL_V(\mathbb{Z}/r\mathbb{Z}) 
\rightarrow \SL_V(\mathbb{Z}/q\mathbb{Z})$ to 
which the former descends. 

\begin{thm} \label{CSPThm}
Let $3 \leq \lvert V \rvert < \infty$, 
let $Q$ be a finite group and let 
$\phi : \SL_V(\mathbb{Z}) \rightarrow Q$ be a homomorphism. 
Then there exists $q \in \mathbb{N}$ and a homomorphism 
$\psi : \SL_V(\mathbb{Z}/q\mathbb{Z}) \rightarrow Q$ 
such that $\phi = \psi \circ \pi_q$. 
\end{thm}

The \emph{mod-$q$ principal congruence subgroup} 
of $\SL_V(\mathbb{Z})$ is: 
\begin{center}
$K_V (q) = \ker \big( \pi_q : \SL_V(\mathbb{Z}) \rightarrow 
\SL_V(\mathbb{Z}/q\mathbb{Z}) \big) = \SL_V(\mathbb{Z}) \cap \big( I_V + q \cdot\mathbf{M}_V (\mathbb{Z}) \big)$. 
\end{center}

This includes the case $K_V(1) = \SL_V (\mathbb{Z})$. 
We suppliment the CSP with some familiar facts 
about the normal subgroup structure of $\SL_V (\mathbb{Z})$. 

\begin{lem}[Chinese Remainder Theorem] \label{CRTlem}
Let $V$ be finite. 
Let $q_1$, $q_2$ be coprime integers. 
Then $(\pi_{q_1} \times \pi_{q_2}) : \SL_V (\mathbb{Z} / q_1 q_2 \mathbb{Z}) \rightarrow \SL_V (\mathbb{Z} / q_1 \mathbb{Z}) \times \SL_V (\mathbb{Z} / q_2 \mathbb{Z})$ is an isomorphism. 
Moreover for $q^{\prime} \mid q_1$, 
the preimage of $\pi_{q_1} \big(K_V (q^{\prime})\big) \times 1$ under this isomorphism 
is $\pi_q \big( K_V (q^{\prime} q_2)\big)$. 
\end{lem}

\begin{propn} \label{elemnrmlgenprop}
Let $V$ be finite with $\lvert V \rvert \geq 3$. 
Suppose that $q \geq 4$ is composite, and that $p$ is 
a prime divisor of $q$. 
Let $N \vartriangleleft \SL_V (\mathbb{Z})$ with 
$K_V (q) \leq N$.  
Suppose there exists $W \subseteq V$ with $\lvert W \rvert \geq 2$, 
such that $K_W (q/p) \leq N$. 
Then $K_V (q/p) \leq N$. 
\end{propn}

\begin{proof}
We distinguish two cases. 
If $p^2$ does not divide $q$, then by Lemma \ref{CRTlem} 
we have that $K_V (q/p)/K_V (q) \cong \SL_V (\mathbb{Z}/p\mathbb{Z})$ 
is quasisimple (by the assumption on $\lvert V\rvert$), with: 
\begin{center}
$K_W (q/p) K_V (q) / K_V (q) \leq N \cap K_V (q/p) / K_V (q) \vartriangleleft K_V (q/p)/K_V (q)$, 
\end{center}
so $N \cap K_V (q/p) / K_V (q)$ is not central in $K_V (q/p)/K_V (q)$. 
We conclude: 
\begin{center}
$N \cap K_V (q/p) / K_V (q) = K_V (q/p)/K_V (q)$ 
\end{center}
so that $K_V (q/p) \subseteq N$. 

Alternatively, if $p^2$ divides $q$, then 
$K_V (q/p) / K_V (q) \cong (\mathbb{Z}/p\mathbb{Z}) ^ {\lvert V \rvert^2 - 1}$ 
is spanned by 
$\lbrace E_{v,w}(q/p):v \neq w \rbrace\cup \lbrace D_{v,w}(1 + q/p):v\neq w \rbrace$ 
where for $\alpha \in (\mathbb{Z}/q\mathbb{Z})^{\ast}$, 
$D_{v,w} (\alpha)$ acts on the basis $V$ via: 
\begin{center}
$D_{v,w} (\alpha) : v \mapsto \alpha v,w \mapsto \alpha^{-1} w, u \mapsto u$ 
for $v \neq u \neq w$. 
\end{center}
For any $v,w,v',w' \in V$ with $v\neq w$, $v'\neq w'$, 
$E_{v,w}(q/p)$ (respectively $D_{v,w}(1 + q/p)$) 
is conjugate to $E_{v',w'}(q/p)^{\pm 1}$ 
(resp. $D_{v',w'}(1 + q/p)^{\pm 1}$) via 
a permutation matrix. 
Since we have $E_{v,w}(q/p) , D_{v,w}(1 + q/p) \in N / K_V(q)$, 
for some $v,w \in W$, the required claim follows. 
\end{proof}

\begin{propn} \label{elemonlynrmlsubgrpprop}
Let $V$ be finite with $\lvert V \rvert \geq 3$ and suppose $(\lvert V \rvert,q)=1$. 
Let $K_V (q) \leq N \vartriangleleft \SL_V(\mathbb{Z})$. 
\begin{itemize}
\item[(i)] If $\pi_q (N)$ is central 
in $\SL_V(\mathbb{Z}/q\mathbb{Z})$, 
then $\lvert N : K_V (q) \rvert \leq \lvert V \rvert ^{\log_2 q}$; 

\item[(ii)] If $\pi_q (N)$ is non-central 
in $\SL_V(\mathbb{Z}/q\mathbb{Z})$, 
then $q$ has a proper divisor $r$ such that $K_V (r) \leq N$. 

\end{itemize}
\end{propn}

\begin{proof}
Let $q = p_1 ^{a_1} \cdots p_l ^{a_l}$ be the prime factorization of $q$. 
By Lemma \ref{CRTlem}, 
\begin{equation*}
\SL_V(\mathbb{Z}/q\mathbb{Z}) 
\cong \prod_{i=1} ^l \SL_V(\mathbb{Z}/p_i ^{a_i} \mathbb{Z}),
\end{equation*}
with the projection of $\pi_q (N)$ to the $i$th factor on the right 
being precisely $\pi_{p_i ^{a_i}} (N)$. 

For (i), note that $l \leq \log_2 q$, so it suffices to show that 
$Z_i = Z \big( \SL_V(\mathbb{Z}/p_i ^{a_i} \mathbb{Z}) \big)$ 
has order at most $\lvert V \rvert$. 
This last claim can be deduced from Hensel's Lemma: 
since $Z_i$ consists of scalar matrices, 
its elements correspond to solutions to the equation $x^{\lvert V \rvert}=1$ 
in $\mathbb{Z}/p_i ^{a_i} \mathbb{Z}$, 
and these are distinct modulo $p_i$, as $\lvert V\rvert$ and $p_i$ 
are coprime. 

For (ii), by Lemma \ref{CRTlem} there exists $i$ such that $\pi_{p_i ^{a_i}} (N)$ 
is non-central in $\SL_V(\mathbb{Z}/p_i ^{a_i} \mathbb{Z})$. 
By projecting to this factor we may therefore suppose that $q = p_i ^{a_i}$ 
and our claim is that $\pi_q (K_V (q/p_i)) \leq \pi_q (N)$. 
This follows from the description of normal subgroups of 
$\SL_V(\mathbb{Z}/p^a \mathbb{Z})$ given, for example, 
in \cite{Kling} (see Theorem 3 and the Remarks following), 
noting that every nonzero ideal of $\mathbb{Z}/p^a \mathbb{Z}$ 
contains $p^{a-1}\mathbb{Z}/p^a \mathbb{Z}$. 
\end{proof}

Finally, we record an easy number-theoretic observation. 

\begin{lem} \label{notheorylem}
Let $m \geq 8$ and $q \in \mathbb{N}$. 
Then there exists a prime $p \in [m/4,m]$ and $r \mid q$ 
such that $(p,r)=1$ and $r \geq \sqrt{q}$. 
\end{lem}

\begin{proof}
Recall Bertrand's Postulate: for any $n \geq 4$, 
there exists a prime number in $(n/2,n)$. 
So let $\tilde{p}$ be a prime in $(m/2,m)$. 
Let $a \in \mathbb{N}$ be maximal such that $\tilde{p}^a \mid q$. 
If $\tilde{p}^a \leq \sqrt{q}$ then take $\tilde{p} = p$ 
and $r = q/\tilde{p}^a$. 
If $\tilde{p}^a > \sqrt{q}$ then let $p$ be any 
prime in $(m/4,m)$ with $p \neq \tilde{p}$ 
(such $p$ exist; for instance any prime in $(m/4,m/2)$ will do). 
Then let $b \in \mathbb{N}$ be maximal such that $p^b \mid q$, 
and let $r = q / p^b$. 
\end{proof}

We now prove Theorem \ref{elemthm}. 
The key to the upper bound in (\ref{elemthmgeneqn}) 
is the following construction. 
Henceforth we write $B(n) = B_{S,\omega_0} (n)$. 

\begin{propn} \label{elemUBprop}
Let $A(n) = \lbrace g \in E_{B(n)} (\mathbb{Z})
: \lVert g \rVert_{\infty} \leq 2^n \rbrace$; 
let $I(n) = A(n) \cdot B_S (n) 
\subseteq \mathcal{E}(\Gamma,\Omega,\mathbb{Z})$, 
and let $(\rho_n , \theta_n)$ be a local embedding 
of $\big( B_S(n) , B(n) \big)$ into $(Q_n,X_n)$. 
Then for any $q_n > 2^{n+1}$ there is a local embedding of $I(n)$ 
into $\mathcal{E}(Q_n,X_n,\mathbb{Z}/q_n \mathbb{Z})$. 
\end{propn}

\begin{proof}
Let $\tilde{\theta}_n :E_{B(n)}(\mathbb{Z})\rightarrow E_{X_n}(\mathbb{Z})$ be the injection induced by 
$\theta_n :B(n)\rightarrow X_n$. 
Note that for $A \in E_{B(n)}(\mathbb{Z})$, 
$\Vert \tilde{\theta}_n(A) \Vert_{\infty} = \Vert A \Vert_{\infty}$. 

Let $\Phi_n : I(n) \rightarrow \mathcal{E}(Q_n,X_n,\mathbb{Z})$ 
be given by $\Phi_n(A \cdot g) = \tilde{\theta}_n(A)\cdot \rho_n(g)$. 
We claim that $\Phi_n$ is a local embedding. 
Given this claim, the conclusion easily follows: 
the congruence homomorphism 
$\pi_{q_n} : E_{X_n} (\mathbb{Z}) \rightarrow E_{X_n} (\mathbb{Z}/q_n \mathbb{Z})$ extends to a homomorphism 
$\pi_{q_n}:\mathcal{E}(Q_n,X_n,\mathbb{Z}) \rightarrow 
\mathcal{E}(Q_n,X_n,\mathbb{Z}/q_n \mathbb{Z})$ 
(which is the identity map on $Q_n$). 
Then $\pi_{q_n} \circ \Phi_n: I(n) \rightarrow 
\mathcal{E}(Q_n,X_n,\mathbb{Z}/q_n \mathbb{Z})$ 
is the desired local embedding: 
if  $A_1,A_2 \in A(n)$ and $g_1,g_2 \in B_S(n)$, with: 
\begin{align*}
\pi_{q_n}(\tilde{\theta}_n(A_1))\cdot \rho_n (g_1)
&=(\pi_{q_n} \circ \Phi_n)(A_1 \cdot g_1) \\
&=(\pi_{q_n} \circ \Phi_n)(A_2 \cdot g_2) \\
&= \pi_{q_n}(\tilde{\theta}_n(A_2))\cdot \rho_n (g_2)
\end{align*}
then $g_1 = g_2$ and $\pi_{q_n}(\tilde{\theta}_n(A_1)) 
= \pi_{q_n}(\tilde{\theta}_n(A_2))$. 
The latter implies $\tilde{\theta}_n(A_1)= \tilde{\theta}_n(A_2)$ 
so that $A_1 = A_2$, 
since $\Vert \tilde{\theta}_n(A_i) \Vert_{\infty} \leq 2^n$ 
and $q_n > 2^{n+1}$. 

Therefore let $A_1,A_2 \in A(n)$ and $g_1,g_2 \in B_S(n)$. 
Then $(A_1\cdot g_1)(A_2\cdot g_2) 
= \big( A_1 A_2 ^{g_1 ^{-1}} \big)\cdot (g_1 g_2) \in I(n)$ 
iff (i) $g_1 g_2 \in B_S(n)$; 
(ii) $A_1 A_2 ^{g_1 ^{-1}} \in E_{B(n)}(\mathbb{Z})$ 
and (iii) $\Vert A_1 A_2 ^{g_1 ^{-1}} \Vert_{\infty} \leq 2^n$. 
Since $A_1 \in E_{B(n)}(\mathbb{Z})$, (ii) 
is equivalent to 
(ii') $A_2 \in \big( E_{B(n)}(\mathbb{Z}) \big)^{g_1} 
= E_{B(n)g_1}(\mathbb{Z})$. Under these conditions: 
\begin{align*}
\Phi_n \big( (A_1 g_1)(A_2 g_2) \big)
&= \tilde{\theta}_n (A_1)\tilde{\theta}_n (A_2 ^{g_1 ^{-1}})
\rho_n(g_1)\rho_n(g_2)\\
\Phi_n (A_1 g_1)\Phi_n (A_2 g_2)
&=\tilde{\theta}_n (A_1)\tilde{\theta}_n (A_2)^{\rho_n(g_1) ^{-1}}
\rho_n(g_1)\rho_n(g_2)
\end{align*}
so, writing $\tau_1 = \tilde{\theta}_n (A_2)^{\rho_n(g_1) ^{-1}}$ 
and $\tau_2 = \tilde{\theta}_n (A_2 ^{g_1 ^{-1}})$, 
it suffices to show that $\tau_1 = \tau_2$. 
We now conclude exactly as in the proof of Proposition \ref{RFextnprop}, 
comparing the effect of $\tau_1$ and $\tau_2$ on the standard basis 
for the module $\mathbb{Z}^{X_n}$. 
\end{proof}

There is also a variant of the construction of Proposition \ref{elemUBprop} 
for $\mathcal{E}(\Gamma,\Omega,\mathbb{F}_p)$. 

\begin{propn} \label{p-elemUBprop}
Let $A(n) = E_{B(n)} (\mathbb{F}_p)$; 
let $I(n) = A(n) \cdot B_S (n) 
\subseteq \mathcal{E}(\Gamma,\Omega,\mathbb{F}_p)$, 
and let $(\rho_n , \theta_n)$ be a local embedding 
of $\big( B_S(n) , B(n) \big)$ into $(Q_n,X_n)$. 
Then there is a local embedding of $I(n)$ 
into $\mathcal{E}(Q_n,X_n,\mathbb{F}_p)$. 
\end{propn}

\begin{proof}
This is much the same as the proof of Proposition \ref{elemUBprop}, 
but easier: let $\tilde{\theta}_n :E_{B(n)}(\mathbb{F}_p)\rightarrow E_{X_n}(\mathbb{F}_p)$ be the injection induced by 
$\theta_n :B(n)\rightarrow X_n$, 
and let $\Phi_n : I(n) \rightarrow \mathcal{E}(Q_n,X_n,\mathbb{F}_p)$ 
be given by $\Phi_n(A \cdot g) = \tilde{\theta}_n(A)\cdot \rho_n(g)$. 
Then $\Phi_n$ is the desired local embedding, 
which is proved exactly as in Proposition \ref{elemUBprop}. 
\end{proof}

\begin{proof}[Proof of Theorem \ref{elemthm}]
First we prove (\ref{elemthmgeneqn}). 
Let $A(n)$, $I(n)$ and $q_n$ be as in Proposition \ref{elemUBprop}. 
We show that $B_{S \cup T(\omega_0)} (n) \subseteq I(n)$. 
The upper bound in (\ref{elemthmgeneqn}) follows, as, 
taking $q_n \leq \exp(n)$ (which we may) we have: 
\begin{align*}
\lvert \mathcal{E}(Q_n,X_n,\mathbb{Z}/q_n \mathbb{Z}) \rvert
&\leq \lvert \SL_{X_n} (\mathbb{Z}/q_n \mathbb{Z}) \rvert \cdot \lvert Q_n \rvert \\
& \leq q_n ^{\lvert X_n \rvert^2} \cdot \lvert Q_n \rvert \\
& \leq \exp \big( n \lvert X_n \rvert^2 \big) \lvert Q_n \rvert\text{.}
\end{align*}
Letting $C(S,T(\omega_0),n)$ be as in Lemma \ref{SDprodgenlem}, 
it is enough to check that 
\linebreak$B_{C(S,T(\omega_0),n)}(n)\subseteq A(n)$. 
For $h \in C(S,T(\omega_0),n)$ there exist 
$s \in S$ and $g \in B_S (n-1)$ 
such that $h=E_{\omega_0,\omega_0 s} ^g=E_{\omega_0 g,\omega_0 sg}$ 
or $h=E_{\omega_0 s,\omega_0 } ^g=E_{\omega_0 sg,\omega_0 g}$. 
In either case, $h \in E_{B (n)} (\mathbb{Z})$, 
so $B_{C(S,T(\omega_0),n)} (m) \subseteq E_{B(n)}(\mathbb{Z})$ 
for all $m \in \mathbb{N}$. 
We claim, by induction, that for all $m \in \mathbb{N}$ 
and $g \in B_{C(S,T(\omega_0),n)}(m)$, 
$\Vert g \Vert_{\infty} \leq 2^{m-1}$. 
This is clear for $m=1$. 
Letting $g \in B_{C(S,T(\omega_0),n)}(m)$ 
and supposing the claim holds for smaller $m$, 
write $g = g^{\prime} h$ for 
$g^{\prime} \in B_{C(S,T(\omega_0),n)}(m-1)$ and 
$h \in C(S,T(\omega_0),n)$, and let $v \in \Omega$. 
If $v \notin B(n)$ then $vg=v$ whereas, if $v \in B(n)$ 
then by induction: 
\begin{equation*}
v g^{\prime} = \sum_{u \in B(n)} \lambda_u u
\end{equation*}
for some $\lambda_u \in \mathbb{Z}$ with 
$\lvert \lambda_u \rvert \leq 2^{m-2}$, 
so that there exist $w_1 , w_2 \in B(n)$ such that 
$h^{\pm 1} = E_{w_1,w_2}$ and: 
\begin{equation*}
vg = (v g^{\prime})h 
=(\pm\lambda_{w_1} +\lambda_{w_2}) w_2 +\sum_{u \neq w_2} \lambda_u u\text{.}
\end{equation*}
Hence $\Vert vg \Vert_{\infty} \leq 2^{m-1}$, 
and since this holds for all $v$, 
$\Vert G \Vert_{\infty}\leq 2^{m-1}$, as required. 

For the lower bound we apply Proposition \ref{LERFRFsubgrp} 
with $\mathcal{E}(\Gamma,\Omega,\mathbb{Z})$ 
replacing $\Gamma$ and $S \cup T(\omega_0)$ replacing $S$. 
Let $V_n = B_{S,\omega_0}(n)$, 
let $H_n = E_{V_n}(\mathbb{Z}) = \SL_{V_n} (\mathbb{Z})$, 
let $G_n = (V_n,E_n)$ be as in Lemma \ref{permtreediamlem}, 
let $S^{\prime} _n = \lbrace E_{v,w}:v,w \in V_n , v\neq w \rbrace$ 
and let $S^{\prime\prime} _n = B_{S\cup T(\omega_0)} (4Cn^2)\cap H_n$, 
where $C>0$ is as in Proposition \ref{elemFGprop}. 
This is permissable, 
since for $v,w \in V_n$ there exist $g,h\in B_S (n)$ with 
$v = \omega_0 g$ and $w = \omega_0 h$, 
so, using (\ref{elemFGeqn}), and by the relation: 
\begin{center}
$E_{v,w} = \big[ E_{v,\omega_0},E_{\omega_0,w} \big]$ 
(for $\omega_0 , v$ and $w$ all distinct), 
\end{center}
$S^{\prime} _n \subseteq S^{\prime\prime} _n$, 
as per the conditions of Proposition \ref{LERFRFsubgrp}. 
Since $S^{\prime\prime} _n \subseteq B_S (4Cn^2)$, 
we may clearly take $f(n) = 4Cn^2$. 
Taking $V=V_n$ in Theorem \ref{SLpresnthm} we have $S^{\prime} _n = S_G$, 
so $H_n$ admits a finite presentation $\langle S^{\prime} _n \mid R_n \rangle$, 
with the relations $R_n$ being words of bounded length. 
We may therefore take $g(n)$ to be a constant in 
Proposition \ref{LERFRFsubgrp}, and conclude that 
for constants $C_2 \geq C_1 > 0$: 
\begin{equation} \label{elemLEFRFineq}
\mathcal{L}_{\mathcal{E}(\Gamma,\Omega,\mathbb{Z})} ^{S \cup T(\omega_0)} (C_2 n^2) \geq 
\mathcal{R}_{\SL_{V_n}(\mathbb{Z})} ^{S_n ^{\prime\prime}}(C_1)
\end{equation}
(where $C_1 > 0$ may be chosen sufficiently large for our purposes). 
It therefore suffices to give a lower bound 
on the right-hand side of the last inequality. 
To this end let $Q$ be a finite group, 
and let $\phi : \SL_{V_n} (\mathbb{Z}) \rightarrow Q$ 
be an epimorphism whose restriction to 
$B_{S_n ^{\prime\prime}} (C_1)$ is injective. 
For $n \geq 2$, $\lvert V_n \rvert \geq 3$, 
so by Theorem \ref{CSPThm} 
there exists $q \in \mathbb{N}$ and a homomorphism 
$\psi : \SL_{V_n}(\mathbb{Z}/q\mathbb{Z}) \rightarrow Q$ 
such that $\phi = \psi \circ \pi_q$, 
where $\pi_q : \SL_{V_n}(\mathbb{Z}) 
\rightarrow \SL_{V_n}(\mathbb{Z}/q\mathbb{Z})$ 
is the congruence map. 
We assume $q$ to be the minimal natural number 
for which such a homomorphism $\psi$ exists. 

First note that, for $m$ larger than an absolute constant, 
$B_S (m)$ contains $S^{\prime} _2$, 
a generating set for $\SL_{V_2}(\mathbb{Z})$. 
Thus $B_{S^{\prime} _2} (4Cn^2/m) \subseteq S_n ^{\prime\prime}$, 
so $B_{S^{\prime} _2} (4C C_1 n^2/m) 
\subseteq B_{S_n ^{\prime\prime}} (C_1)$. 
Now $\SL_{V_2}(\mathbb{Z})$ is a group of exponential growth, 
so there exists an absolute constant $C_3 > 0$ such that 
$\lvert B_{S_n ^{\prime\prime}} (C_1 ) \cap \SL_{V_2}(\mathbb{Z}) \rvert \geq \exp(C_3 n^2)$. 
Now $\pi_q(\SL_{V_2}(\mathbb{Z}))=\SL_{V_2}(\mathbb{Z}/q\mathbb{Z})$, 
so by the injectivity of 
$\pi_q \mid_{B_{S_n ^{\prime\prime}} (C_1)}$ 
(and $\lvert V_2 \rvert \leq 100 \lvert S \rvert^2$), 
\begin{center}
$\exp(C_3 n^2) \leq \lvert \SL_{V_2}(\mathbb{Z}/q\mathbb{Z}) \rvert 
\leq q^{100 \lvert S \rvert^2}$. 
\end{center}
Hence there is a constant $C_4 > 0$ such that $q \geq \exp(C_4 n^2)$. 
For $n$ larger than an absolute constant, 
$\lvert V_n \rvert \geq 8$, 
so by Lemma \ref{notheorylem} there exists 
a prime $\lvert V_n \rvert / 4 \leq p \leq \lvert V_n \rvert$ 
such that there exist $a , r \in \mathbb{N}$ with $q = p^a r$, 
$(p,r)=1$ and $r \geq \sqrt{q} \geq \exp((C_4 / 2) n^2)$. 
Let $W_n \subseteq V_n$ with $\lvert W_n \rvert = p$. 
Now let $N = \ker (\psi) 
\vartriangleleft \SL_{V_n}(\mathbb{Z}/q\mathbb{Z})$. 
By minimality of $q$, 
there is no proper divisor $q^{\prime}$ of $q$ such that 
$\pi_q (K_{V_n}(q^{\prime})) \leq N$. 
By Lemma \ref{CRTlem} we may naturally identify 
$\SL_{V_n}(\mathbb{Z}/q\mathbb{Z})$ 
with $\SL_{V_n}(\mathbb{Z}/p^a\mathbb{Z}) \times \SL_{V_n}(\mathbb{Z}/r\mathbb{Z})$. 
We claim that the restriction of $\psi$ to 
$\SL_{W_n}(\mathbb{Z}/r\mathbb{Z})$ is injective. 
This yields the desired conclusion, since: 
\begin{equation}
\lvert Q \rvert \geq \lvert \SL_{W_n}(\mathbb{Z}/r\mathbb{Z}) \rvert 
\geq r^{C_5 \lvert W_n \rvert^2} 
\geq q^{C_5 \lvert V_n \rvert^2 /32} 
\geq \exp \big( C_6 n^2 \lvert V_n \rvert \big)
\end{equation}
for $n$ sufficiently large and $C_5 , C_6 > 0$ absolute constants. 
Taking $Q$ of minimal order and applying (\ref{elemLEFRFineq}), 
we have $\mathcal{L}_{\mathcal{E}(\Gamma,\Omega,\mathbb{Z})} ^{S \cup T} (C_2 n^3)
\geq C_5 \exp \big( C_6 n^3 \lvert V_n \rvert \big)$. 

Suppose, then, that $N \cap \SL_{W_n} (\mathbb{Z}/r\mathbb{Z})\neq 1$. 
By Proposition \ref{elemonlynrmlsubgrpprop}, 
there exists a prime divisor $p^{\prime}$ of $r$ 
such that: 
\begin{center}
$K_{W_n} (r/p^{\prime}) \leq N\cap \SL_{W_n} (\mathbb{Z}/r\mathbb{Z})$. 
\end{center}
Pulling back under the isomorphism from Lemma \ref{CRTlem}, we have 
$K_{W_n} (q/p^{\prime})\leq N\leq\SL_{V_n}(\mathbb{Z}/q\mathbb{Z})$. 
By Proposition \ref{elemnrmlgenprop}, 
$K_{V_n} (q/p^{\prime}) \leq N$, 
so $\phi$ factors through $\pi_{q/p^{\prime}}$, 
contradicting minimality of $q$. 

The proof of (\ref{p-elemthmgeneqn}) is much the same, but easier. 
For the upper bound, let $I(n)$ be as in Proposition \ref{p-elemUBprop}. 
We show that $B_{S\cup T(\omega_0)}(n) \subseteq I(n)$ 
just as in the integer case (except without any concerns about 
the norms of matrices). 
By Proposition \ref{p-elemUBprop} we then have a local embedding of 
$B_{S\cup T(\omega_0)}(n)$ into $\mathcal{E}(Q_n,X_n,\mathbb{F}_p)$, 
which has order at most $p^{\lvert X_n \rvert ^2} \lvert Q_n \rvert$. 
For the lower bound we may once again take 
$S^{\prime} _n = \lbrace E_{v,w}:v,w \in V_n , v\neq w \rbrace$ 
and $S^{\prime\prime} _n = B_{S\cup T(\omega_0)} (4Cn^2)\cap H_n$, 
with $H_n = \SL_{V_n} (\mathbb{F}_p)$, and deduce: 
\begin{equation*} 
\mathcal{L}_{\mathcal{E}(\Gamma,\Omega,\mathbb{Z})} ^{S \cup T(\omega_0)} (C_2 n^2) \geq 
\mathcal{R}_{\SL_{V_n}(\mathbb{F}_p)} ^{S_n ^{\prime\prime}}(C_1)
\end{equation*}
for $C_2 > C_1 > 0$, where we are able to choose $C_1$ sufficiently large. 
Since $\SL_{V_n}(\mathbb{F}_p)$ is quasisimple whenever $\lvert V_n \rvert\geq 4$, 
we have: 
\begin{equation*} 
\mathcal{R}_{\SL_{V_n}(\mathbb{F}_p)} ^{S_n ^{\prime\prime}}(C_1) 
= \lvert \PSL_{V_n}(\mathbb{F}_p) \rvert \geq c p^{\lvert V_n \rvert^2}
\end{equation*}
for some $c>0$ depending only on $p$. 
\end{proof}

\begin{rmrk}
\normalfont
LEF for $\mathcal{E}(\Gamma,\Omega,R)$ also holds for rings other 
than $\mathbb{Z}$ or $\mathbb{F}_p$. Indeed, for $\Gamma,\Omega$ satisfying 
the conditions of Theorem \ref{elemthm}, 
$\mathcal{E}(\Gamma,\Omega,R)$ is a LEF group whenever 
$R$ is a \emph{LEF ring}. 
We recall that a ring $R$ is LEF if, 
for every finite subset $F \subseteq R$ there 
is a finite ring $A$ and an injective map $\phi : F \rightarrow A$ 
such that for all $r,s \in F$, $r+s \in F$ implies 
$\phi(r+s)=\phi(r)+\phi(s)$ and $rs \in F$ implies 
$\phi(rs) = \phi(r)\phi(s)$. 
Apart from $\mathbb{Z}$, LEF rings include: 
\begin{itemize}
\item[(i)] Any finite ring; 

\item[(ii)] The group-ring $R[\Delta]$, 
where $R$ is a LEF ring 
and $\Delta$ is a LEF group. 

\end{itemize}
Rings of type (ii) include the polynomial ring 
$R[t_1 ^{\pm 1} , \ldots t_k ^{\pm 1}]$ 
(respectively $R\langle t_1 ^{\pm 1} , \ldots t_k ^{\pm 1}\rangle$) 
over $R$ in $k$ commuting (respectively non-commuting) 
variables $t_1 , \ldots , t_k$ and their multiplicative inverses, 
obtained by taking $\Delta = \mathbb{Z}^k$ 
(respectively $\Delta = F_k$). 

Moreover in many cases we can prove an inequality redolent of the 
upper bound in (\ref{elemthmgeneqn}). 
First we require the group $\mathcal{E}(\Gamma,\Omega,R)$ to be 
finitely generated. 
This will be the case, for instance, if there are finite 
subsets $T \subseteq R$, $U \subseteq R^{\ast}$, 
with $1 \in U$, such that $R$ is generated as a $\mathbb{Z}$-module 
by: 
\begin{center}
$\lbrace u_1 \cdots u_k t u_1 ^{\prime} \cdots u_k ^{\prime} 
: k \in \mathbb{N}; t \in T; u_i,u_i ^{\prime} \in U \rbrace$
\end{center}
(so that in particular, $R$ is generated as a ring by 
$T \cup U \cup U^{-1}$). 
This condition is satisfied in the examples (i) and (ii) above 
(provided, in case (ii), the coefficient ring satisfies it too). 
To see that our condition implies finite generation of 
$\mathcal{E}(\Gamma,\Omega,R)$, 
note that for $r \in R$; $u \in R^{\ast}$ 
and $v,w,x \in \Omega$ distinct: 
\begin{center}
$h_{v,x}(u) E_{v,w}(r) h_{v,x}(u)^{-1} = E_{v,w}(ur)$ 
and $h_{w,x}(u)^{-1} E_{v,w}(r) h_{w,x}(u) = E_{v,w}(ru)$, where $h_{y,z}(u) = E_{y,z} (u)E_{z,y} (-u^{-1})E_{y,z} (u)E_{y,z} (1)E_{z,y} (-1)E_{y,z} (1)$
\end{center}
Now we require a notion of LEF growth for rings. 
For $S \subseteq R$ a finite set generating $R$ as a ring, 
with $0,1 \in S$, define $B_S (n)$ inductively by $B_S(1) = S$ 
and $B_S (n+1) = \lbrace sx+y : s\in S;x,y \in B_S(n) \rbrace$. 
The LEF growth $\mathcal{L}_R ^S (n)$ of $R$ with respect to $S$ 
is then the minimal size of a finite ring admitting 
a local embedding of $B_S (n)$. 
We note that, for $v , v_i , w_i \in \Omega$, and $s_i \in S$, 
\begin{equation*}
v E_{v_1,w_1} (s_1) \cdots E_{v_n,w_n} (s_n) 
= \lambda_0 v + \sum_i \lambda_i w_i
\end{equation*}
for some $\lambda_i \in B_S(n)$. 
Arguing as in Theorem \ref{elemthm}, we can conclude: 
\begin{equation} \label{elemgenReqn}
\mathcal{L}_{\mathcal{E}(\Gamma,\Omega,R)} (n) 
\preceq \mathcal{L}_R (n)^{\lvert X_n \rvert^2} \lvert Q_n \rvert\text{.}
\end{equation}
Finally, we can often usefully bound $\mathcal{L}_R (n)$ 
for finitely generated LEF rings $R$. 
For instance $\mathcal{L}_{(\mathbb{Z},+,\cdot)} (n)\approx\exp (n)$; 
if $R$ is finite then $\mathcal{L}_R (n)$ is bounded, 
while, if $R$ is a LEF ring generated by the finite set $X$ 
and $\Delta$ is a LEF group generated by the finite set $S$, 
then $R[\Delta]$ is generated by $X \cup S$, 
and $B_{X \cup S} (n)$ consists of combinations 
with coefficients in $B_X(n)$ and supported on $B_S(n)$. 
As such, if $B_X(n)$ and $B_S(n)$ admit local embeddings 
into $A$ and $Q$, respectively, 
then $B_{X \cup S} (n)$ admits a local embedding into $A[Q]$, 
and consequently 
$\mathcal{L}_{R[\Delta]} ^{X \cup S} (n) 
\leq \mathcal{L}_R ^X (n) ^{\mathcal{L}_{\Delta} ^S (n)}$. 
\end{rmrk}

\section{The diversity of LEF growth functions} \label{SpectrumSect}

\begin{defn} \label{permissdefn}
A strictly increasing function $f : \mathbb{N} \rightarrow \mathbb{N}$ 
is \emph{permissible} if (a) $f(0) = 1$; 
(b) for all $n \geq 1$, 
\begin{equation} \label{permissineq}
(f(n)-f(n-1))/2 \leq f(n+1)-f(n) \leq 2\big(f(n)-f(n-1)\big)\text{.}
\end{equation}
\end{defn}

Roughly, a permissable function is one which grows at least linearly, 
at most exponentially (with exponent $2$), 
and whose growth does not oscillate too quickly. 

Let $F_2$ be the free group of rank $2$, 
and let $S = \lbrace a,b \rbrace \subseteq F_2$ be a free basis. 

\begin{propn} \label{PermissProp}
For any permissible function $f$, 
there is a set $\Omega (f)$; 
a transitive action $F_2 \curvearrowright \Omega(f)$; 
a point $\omega_0 \in \Omega(f)$; 
finite groups $Q_n$ and transitive actions 
$Q_n \curvearrowright X_n$, such that for all $n \in \mathbb{N}$, 
\begin{itemize}
\item[(i)] $B_{S,\omega_0} (n) = f(n)$; 
\item[(ii)] There exists a local embedding of 
$\big( B_S(n),B_{S,\omega_0} (n) \big)$ into $(Q_n,X_n)$; 
\item[(iii)] $\lvert X_n \rvert = f(2n)$; 
\item[(iv)] $\lvert Q_n \rvert \leq f(2n)! \exp (Cn)$, 
for some absolute constant $C>0$. 
\end{itemize}
\end{propn}

\begin{proof}[Proof of Theorem \ref{SpectrumMainThm}]
Let $F_2 \curvearrowright \Omega(f)$; $\omega_0 \in \Omega(f)$ 
and $Q_n \curvearrowright X_n$ be as in Proposition \ref{PermissProp}. 
Let $\Delta (f) = \mathcal{S} (F_2 , \Omega(f))$. 
Then by the inequalities (\ref{RFextneqn1}) from Theorem \ref{RFextnThm}, 
and by Proposition \ref{PermissProp} (ii), 
\begin{equation*}
\gamma_{F_2,\Omega(f)} (n)! \preceq \mathcal{L}_{\Delta (f)} (n) 
\preceq \lvert X_n \rvert ! \lvert Q_n \rvert.
\end{equation*}
On the one hand, by Proposition \ref{PermissProp} (i), 
\begin{equation*}
f(n)! \preceq \gamma_{F_2,\Omega(f)} (n)!. 
\end{equation*}
On the other hand, by Proposition \ref{PermissProp} (iii) and (iv), 
\begin{align*}
\lvert X_n \rvert ! \lvert Q_n \rvert & \leq f(2n)! f(2n) \exp (Cn) \\
& \preceq \exp \big( f(2n) \log f(2n)+ \log f(2n) + Cn \big)  \\
& \preceq \exp\big( f(n)\log f(n) \big) \text{ (since $n \preceq f(n)$)}\\
& \preceq f(n)! 
\end{align*}
(using Stirling's approximation). 
\end{proof}

It remains to prove Proposition \ref{PermissProp}. 
First we construct the set $\Omega (f)$ and the action of $F_2$ upon it. 

\begin{propn} \label{PermissActionProp}
Let $f : \mathbb{N} \rightarrow \mathbb{N}$ be a permissable function. 
There exists a set $\Omega(f)$; 
a transitive action $F_2 \curvearrowright \Omega(f)$; 
and a point $\omega_0 \in \Omega(f)$ such that 
for all $n \in \mathbb{N}$, $B_{S,\omega_0} (n) = f(n)$ 
and moreover, letting $L_0 = \lbrace \omega_0 \rbrace$ and for $n \geq 1$ taking 
$L_n = B_{S,\omega_0} (n) \setminus B_{S,\omega_0} (n-1)$, we have: 
\begin{itemize}
\item[(i)] $a$ preserves $L_{2m} \sqcup L_{2m+1}$ 
for each $m \geq 0$; 
\item[(ii)] $b$ fixes $\omega_0$ and preserves 
$L_{2m-1} \sqcup L_{2m}$ for each $m \geq 1$. 
\end{itemize}
\end{propn}

\begin{proof}
Let $\lbrace L_n : n \in \mathbb{N} \rbrace$ be a family of disjoint 
finite sets, with $L_0 = \lbrace \omega_0 \rbrace$ 
and $\lvert L_n \rvert = f(n)-f(n-1)$ for $n \geq 1$. 
Set $\Omega(f)$ to be the disjoint union of the $L_n$. 

Let us define the action of $a$ on $L_{2m} \sqcup L_{2m+1}$. 
Write $\lbrace L_{2m} , L_{2m+1} \rbrace = \lbrace L,L' \rbrace$ 
with $\lvert L \rvert \geq \lvert L' \rvert$. 
By permissability of $f$, there exist $u , v \in \mathbb{N}$ 
with $\lvert L \rvert = 2u +v$ and $\lvert L' \rvert = u+v$. 
Choose partitions $L = U \sqcup V$ and $L' = U' \sqcup V'$ 
with $\lvert U \rvert = 2u$; $\lvert U' \rvert = u$ 
and $\lvert V \rvert = \lvert V' \rvert = v$. 
Let $a$ act on $U \sqcup U'$ as a product of $u$ disjoint $3$-cycles, 
each consisting of two elements of $U$ and one element of $U'$. 
Let $a$ act on $V \sqcup V'$ as a product of $v$ disjoint transpositions, 
each swapping an element of $V$ with an element of $V'$. 
We define the action of $b$ precisely the same way, 
with $L_{2m-1}$ replacing $L_{2m+1}$ throughout the above 
(and $\omega_0$ being fixed by $b$). 

Certainly $a$ and $b$ act as well-defined permutations of $\Omega(f)$ 
satisfying (i) and (ii). 
By construction, for $n \geq 1$ every element of $L_n$ is adjacent in the Schreier 
graph to at least one element each from $L_{n-1}$ and $L_{n+1}$, 
but to no points outside $L_{n-1} \cup L_n \cup L_{n+1}$. 
Thus the elements of $L_n$ are precisely the vertices of 
the Schreier graph at distance $n$ from $\omega_0$. 
\end{proof}

%\begin{lem} \label{SL2FreeLem}
%There exist $A , B \in \SL_2 (\mathbb{Z})$ freely generating 
%a subgroup isomorphic to $F_2$. 
%\end{lem}

%\begin{propn}
%Let $T \subseteq \SL_2 (\mathbb{Z})$ be finite. 
%For $p$ a prime, 
%let $pi_p : \SL_2 (\mathbb{Z})\rightarrow \SL_2 (\mathbb{Z}/p\mathbb{Z})$ 
%be reduction modulo $p$. 
%There exists $C=C(T)>0$ such that for all $n \in \mathbb{N}$ and 
%all $p \geq \exp (C n)$, 
%the restriction of $\pi_p$ to $B_T (n)$ is injective. 
%\end{propn}

%\begin{proof}
%Suppose that there exist $X,Y \in B_T (n)$ distinct  with $\pi_p (X) = \pi_p (Y)$. 
%Then $I_2 \neq XY^{-1} \in \ker (\pi_p) \cap B_T (2n)$. 
%Since $XY^{-1} \equiv I_2 \mod p$, $X Y^{-1}$ must have an entry of absolute 
%value at least $p-1$. 
%\end{proof}

\begin{proof}[Proof of Proposition \ref{PermissProp}]
Let $F_2 \curvearrowright \Omega(f)$ and $\omega_0 \in \Omega(f)$ 
be as in Proposition \ref{PermissActionProp}. Suppose $n geq 2$. 
Let $X_n = B_{S,\omega_0} (2n)$; 
let $\theta_n : B_{S,\omega_0} (n) \rightarrow X_n$ be inclusion. 
We define $\alpha_n , \beta_n \in \Sym (X_n)$ such that 
$(x)\alpha_n = (x)a$ for $x \in B_{S,\omega_0} (2n-1)$; 
the restriction of $\alpha_n$ to $L_{2n}$ is an arbitrary permutation of $L_{2n}$ 
and $(x)\beta_n = (x)b$ for $x \in X_n$. 
By the construction of the $L_i$ in Proposition \ref{PermissActionProp}, 
$\alpha_n$ and $\beta_n$ are well-defined permutations of $X_n$. 
Let $P_n = \langle \alpha_n , \beta_n \rangle \leq \Sym (X_n)$. 
There is a unique homomorphism $\phi_n : F_2 \rightarrow P_n$ 
satisfying $\phi(a) = \alpha_n$ and $\phi(b) = \beta_n$. 

By Example \ref{ELEFListEx} (iv) 
there exists a constant $C>0$ and finite groups $G_n$ 
such that for every $n$, $\lvert G_n \rvert \leq \exp(Cn)$ 
and there is a homomorphism $\psi_n : F_2 \rightarrow G_n$ 
whose restriction to $B_S (n)$ is injective. 
Let $Q_n = P_n \times G_n$, 
and let $Q_n$ act on $X_n$ via $(x) (\sigma,g) = (x)\sigma$. 
Then $Q_n$ and $X_n$ satisfy conclusions (iii) and (iv). 
Further, $\pi_n = (\phi_n \times \psi_n) : F_2 \rightarrow Q_n$ 
is a homomorphism whose restriction to $B_S (n)$ is injective. 

We claim that the restriction of $\pi_n$ to $B_S(n)$ and
$\theta_n$ satisfy condition (\ref{LEFactioneqn}) 
of Definition \ref{LEFactiondefn}. 
Given the above, this will imply that $(\pi_n,\theta_n)$ 
yields a local embedding of actions, and will conclude the proof. 
Let $g = w (a,b) \in B_S (n)$ be a word of length at most $n$, 
and let $\omega \in B_{S,\omega_0} (n)$. 
Suppose that $\omega g \in B_{S,\omega_0} (n)$. 
Then there is a path in $\Schr (F_2,\Omega(f),S)$ 
from $\omega$ to $\omega g$ of length at most $n$. 
This path lies entirely inside 
$B_{S,\omega_0} (3n/2) \subseteq B_{S,\omega_0} (2n-1)$ (since $n \geq 2$). 
By construction of $\alpha_n$ and $\beta_n$, 
the induced subgraphs on $B_{S,\omega_0} (2n-1)$ 
in $\Schr (F_2,\Omega(f),S)$ and $\Schr (P_n,X_n,\lbrace \alpha_n,\beta_n\rbrace)$ 
are isomorphic, with $a$-labelled edges (respectively $b$-labelled edges) 
corresponding to $\alpha_n$-labelled edges (respectively $\beta_n$-labelled edges). 
Thus $\theta_n (\omega) \pi_n (g) = \theta_n (\omega) w(\alpha_n,\beta_n) 
= \omega w(\alpha_n,\beta_n) = \omega g = \theta_n (\omega g)$, 
as required. 
\end{proof}

\begin{proof}[Proof of Corollary \ref{SpectrumCoroll}]
In the cases (i), (ii) and (iii) we take $\tilde{f}(n) = n^{\alpha}/\log(n)$, 
$\exp (\log(n)^{\alpha}) / \log(n)^{\alpha}$ 
and $\exp (n^{\beta})/n^{\beta}$, 
so that $\tilde{f}(n) \log \tilde{f}(n) \approx n^{\alpha}$, 
$\exp (\log(n)^{\alpha})$ and $\exp (n^{\beta})$, respectively. 
Then $\tilde{f}$ satisfies condition (\ref{permissineq}) 
for all $n$ sufficiently large, 
and consequently there exists a permissable function $f$ such that 
$f \approx \tilde{f}$. 
Applying Theorem \ref{SpectrumMainThm} to $f$ and using Stirling's approximation, 
there exist $C,c>0$ such that: 
\begin{equation*}
\exp \big( c \tilde{f}(n) \log \tilde{f}(n) \big) 
 \preceq \mathcal{L}_{\Delta(f)} 
  \preceq \exp \big( C \tilde{f}(n) \log \tilde{f}(n) \big)
\end{equation*}
as required. 
\end{proof}

\subsection*{Acknowledgements}

This work was partially supported by ERC grant no. 648329 ``GRANT''. 
I am grateful to Vadim Alexeev, Goulnara Arzhantseva and Masato Mimura 
for enlightening conversations which helped to 
shape my thinking on the subject of this paper. 
Parts of this work were undertaken while the author was 
a visiting fellow at the Hausdorff Research Institute for Mathematics 
in Bonn for the program ``Logic and Algorithms in 
Group Theory'', 
and other parts while attending the program ``$L^2$-invariants and 
their analogues in positive characteristic'' at 
the Instituto de Ciencias Matem\'{a}ticas in Madrid. 
It is my pleasure to thank the organizers 
for putting together these excellent programs, 
and HIM and ICMAT for providing pleasant working conditions.

\end{document}